\documentclass[a4paper,english]{article}
\usepackage{babel}
\usepackage{amsmath, amsfonts, mathtools, mathrsfs, dsfont, empheq, esint, tikz, tikz-cd}
\usetikzlibrary{patterns}
\usepackage[shortlabels]{enumitem}
\usepackage[colorlinks = true, linkcolor = red, citecolor = blue, final]{hyperref}
\usepackage{amsthm, thmtools}
\usepackage{cleveref}

\declaretheorem[
name=Theorem,
Refname={Theorem,Theorems},	
numberwithin=section
]{theorem}

\declaretheorem[
name=Proposition,
Refname={Proposition,Propositions},
sibling=theorem
]{proposition}

\declaretheorem[
name=Lemma,
Refname={Lemma,Lemmas},
sibling=theorem
]{lemma}

\declaretheorem[
name=Corollary,
Refname={Corollary,Corollaries},
sibling=theorem
]{corollary}

\declaretheorem[
name=Claim,
unnumbered
]{claim*}

\declaretheorem[
name=Definition,
style=definition,
Refname={Definition,Definitions},
sibling=theorem
]{definition}

\declaretheorem[
name=Remark,
style=remark,
Refname={Remark,Remarks},
sibling=theorem
]{remark}

\declaretheorem[
name=Remark,
style=remark,
unnumbered
]{remark*}

\newcounter{listCounter}

\newcommand{\IR}{{\mathbb{R}}}
\newcommand{\mcA}{{\mathcal{A}}}
\newcommand{\mcT}{{\mathcal{T}}}
\newcommand{\mcX}{{\mathcal{X}}}
\newcommand{\mcY}{{\mathcal{Y}}}
\newcommand{\mcV}{{\mathcal{V}}}
\newcommand{\mcW}{{\mathcal{W}}} 
\newcommand{\abs}[1]{{\left|{#1}\right|}}
\newcommand{\norm}[1]{{\|#1\|}}
\newcommand{\inprod}[2]{{\left\langle#1,#2\right\rangle}}
\newcommand{\pinprod}[2]{{\left(#1,#2\right)}}
\newcommand{\md}{\mathrm{d}}
\newcommand{\loc}{{\mathrm{loc}}}
\newcommand{\emptyarg}{{\,\cdot\,}}
\newcommand{\Tr}{{\mathrm{Tr}\ }}
\DeclareMathOperator*{\esssup}{ess\,sup}

\newcommand{\solA}{u}	\newcommand{\solAvar}{{\tilde{\solA}}}
\newcommand{\solB}{v}	\newcommand{\solBvar}{{\tilde{\solB}}}
\newcommand{\solC}{w}	\newcommand{\solCvar}{{\tilde{\solC}}}

\newcommand{\testA}{\eta}
\newcommand{\testB}{\psi}
\newcommand{\testC}{\varphi}

\newcommand{\nonlin}{\phi}
\newcommand{\Nonlin}{\Phi}
\newcommand{\source}{f}
\newcommand{\sourcevar}{{\tilde{\source}}}
\newcommand{\LipBoundSource}{L}

\title{Well-posedness of singular-degenerate porous medium type equations and application to biofilm models}
\author{
	Hissink Muller, Victor\footnote{corresponding author}\\
	Radboud Universiteit, IMAPP - Mathematics\\
	PO Box 9010, 6500 GL Nijmegen,
	The Netherlands\\
	\texttt{V.HissinkMuller@math.ru.nl}	
	\and 
	Sonner, Stefanie\\
	Radboud Universiteit, IMAPP - Mathematics\\
	PO Box 9010, 6500 GL Nijmegen,
	The Netherlands\\
 	\texttt{S.Sonner@math.ru.nl} 	
 }
\date{}
\begin{document}
	\maketitle
	
\begin{abstract}
We show the well-posedness for a large class of degenerate parabolic 
equations with an additional singularity and mixed Dirichlet-Neumann boundary conditions on bounded Lipschitz domains.
The proof is based on an $L^1$-contraction result. In addition, we analyze systems where  degenerate equations are coupled to semilinear reaction diffusion equations. 
This setting includes mathematical models for biofilm growth which are the motivation for our analysis.     
\end{abstract}


\noindent\textbf{MathSubjClass:} 35K57 (35K65, 35K67, 35K20, 92C17, 92D25)\\

\noindent Declarations of interest: none.

\section{Introduction}\label{sect:Introduction}
We establish well-posedness results for a large class of second-order quasilinear degenerate parabolic equations with mixed Dirichlet-Neumann boundary conditions. In particular, we prove the existence and an $L^1$-contraction result for solutions of initial-/boundary value problems of the form
\begin{subequations}\label{eq:SD-PME.Problem}
	\begin{empheq}[left=\empheqlbrace]{align}
		\solA_t	&= \Delta\nonlin(\solA)+\source(\emptyarg,\solA)	&\quad &\text{in }\Omega\times(0,T),\label{eq:SD-PME.Problem.DiffEqn}\\
		\solA &=\solA_0											&\quad &\text{in }\Omega\times\{0\},\\
		\nonlin(u)	&= \nonlin(\solA^D)											&\quad &\text{on }\Gamma\times(0,T),\\\label{eq:SD-PME.Problem.Condition.Dirichlet}
		\partial_\nu\nonlin(\solA)		&=0													&\quad &\text{on }\partial\Omega\backslash\Gamma\times(0,T),
	\end{empheq}
\end{subequations}
where the solution $\solA$ takes values in $[0,1)$ and $\nonlin:[0,1)\to\mathbb{R}$ is a strictly increasing function with a degeneracy $\nonlin'(0)=0$ and a singularity $\nonlin(1)=\infty$.
Moreover, $T>0$, $\Omega\subset \IR^N$ is a bounded Lipschitz domain with boundary $\partial \Omega$, and $\Gamma\subset\partial\Omega$ has strictly positive measure and is the part of the domain on which Dirichlet boundary conditions are prescribed.

Furthermore, we consider \Cref{eq:SD-PME.Problem} coupled to a semilinear reaction-diffusion equation and extend the well-posedness theory for coupled systems of the form
\begin{subequations}\label{eq:SD-PME.coupled.Problem}
	\begin{empheq}[left=\empheqlbrace]{align}
		\begin{aligned}	
			\solA_t&=\Delta\nonlin(\solA)+{f}(\emptyarg,\solA,\solB)\\	
			\solB_t&=\Delta\solB+{g}(\emptyarg,\solA,\solB)
		\end{aligned}	&\qquad\text{in }\Omega\times(0,T), \label{eq:coupled.SD-PME.problem.diff.eq}
	\end{empheq}
	with the initial data and mixed boundary conditions
	\begin{empheq}[left=\empheqlbrace]{align}
		\begin{aligned}
			\solA&=\solA_0&& \text{on }{\Omega}\times\{0\},
			&\solB&=\solB_0
			&&\text{on }{\Omega}\times\{0\},\\
			\nonlin(\solA)	&=\nonlin(\solA^D)	& &\text{on}\ \Gamma_1\times(0,T),&\quad \partial_\nu\nonlin(\solA)&=0&&\text{on}\ \partial\Omega\backslash\Gamma_1\times(0,T),\\
			\solB			&=\solB^D			& &\text{on}\ \Gamma_2\times(0,T),&\quad \partial_\nu \solB		&=0&&\text{on}\ \partial\Omega\backslash\Gamma_2\times(0,T),
		\end{aligned}
	\end{empheq}
\end{subequations}
where $\Gamma_1,\Gamma_2\subset\partial\Omega$ have positive measure.

The motivation for our analysis is the biofilm growth  model introduced and numerically studied in \cite{Eberl2000}. Biofilms are dense aggregations of bacterial cells encased in a slimy matrix of extracellular polymeric substances that grow in moist environments, often attached to a surface. On the mesoscale, mature biofilms can show complex heterogeneous spatial structures and mushroom shaped architectures that the model \cite{Eberl2000} is capable to predict.

The biofilm growth model consists of two reaction-diffusion equations for the biomass density $M$ and the growth limiting nutrient concentration $C$. The biofilm and surrounding region are assumed to be continua that are separated by a sharp interface. 
Both model variables are dimensionless, $C$ is scaled with respect to the bulk concentration and $M$ with respect to the maximum biomass density. 
The equations are coupled via the reaction terms, which are Monod functions describing biomass production and nutrient consumption.
While the nutrient is dissolved in the domain and $C$ satisfies a classical semilinear reaction-diffusion equation, the equation for the biomass density is quasilinear with a diffusion coefficient that vanishes as the biomass density approaches zero and blows up as the biomass density approaches its maximum value.
In particular, $C$ and $M$ satisfy the system	
\begin{equation}\label{eq:Biofilm}
	\left\{\begin{aligned}
		\partial_t M&=d_2\nabla\cdot\left(D(M)\nabla M\right)-K_2M+K_3\frac{CM}{K_4+C}&&\text{in}\ \Omega\times(0,T),\\
		\partial_t C&=d_1\Delta C-K_1\frac{CM}{K_4+C}&&\text{in}\ \Omega\times(0,T),
	\end{aligned}\right.
\end{equation}
where $D$ is given by 
\begin{align*}
	D(M)=\frac{M^b}{(1-M)^a}, \qquad a\geq 1, \ b> 0.
\end{align*}
The diffusion coefficient of the nutrient $d_1$ and the biomass motility coefficient $d_2$ are positive, the lysis rate
$K_1$, the maximum specific consumption rate $K_2$ and the 
maximum specific growth rate $K_3$ are non-negative and the half-saturation constant $K_4$ is positive.
Moreover, $\Omega\subset \mathbb{R}^n$, $n=1,2,3,$ is a bounded domain 
and the actual biofilm is the subregion where $M$ is positive, 
\[
\Omega_{M}(t)=\{x\in\Omega\mid M(x,t)>0\},
\]
see \Cref{fig:biofilm.scheme}. Here, the surface on which the 
biofilm grows is the bottom part of the boundary.

\begin{figure}
	\centering
	\begin{tikzpicture}
	\filldraw[pattern=north west lines, pattern color=black!30!white, draw=black] plot[smooth, tension=.7] coordinates {(0,0) (-0.25,2) (0.25,3) (-2,4) (-2,1) (-2.5,0)};
	\filldraw[pattern=north west lines, pattern color=black!30!white, draw=black] plot[smooth, tension=.7] coordinates {(1.25,0) (1,1.5) (2,2) (2.7,1.7) (2.8,1) (2.5,0)};
	\node at (-1.1, 1.8) {$\Omega_M(t)$};
	\node at (2, 0.75) {$\Omega_M(t)$};
	\node at (1.2, 2.3) {$M(t)=0$};
	\node at (-0.6, 2.85) {$M(t)>0$};
	\node at (1.9, 1.55) {$M(t)>0$};
	\node at (-2.7,4.7) {$\Omega$};
	\draw[line width=1pt] (-3,0) -- (3,0) -- (3,5) -- (-3,5) -- cycle;
	\end{tikzpicture}
	
	\caption{The biofilm $\Omega_M$ in the bulk liquid $\Omega$.}
	\label{fig:biofilm.scheme}
\end{figure}
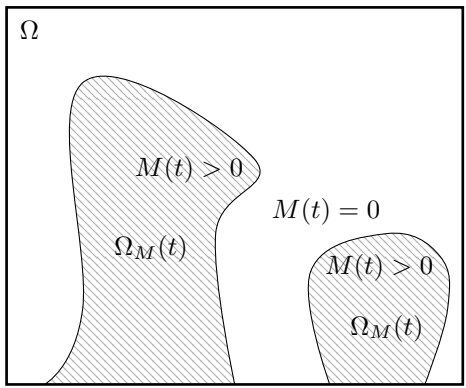

The biomass diffusion coefficient $D$  has a degeneracy as known from the porous medium equation, $\solA_t=\Delta\solA^m$, $m>1$. It ensures a finite speed of interface propagation and a sharp interface between the biofilm and the surrounding liquid.
On the other hand, the singularity in the diffusion coefficient for $M=1$ ensures that spatial spreading becomes very large whenever $M$ approaches $1$, which ensures that the biomass density remains bounded by its maximum value.
In particular, no boundedness assumption is needed for the reaction terms.
Setting 
\[
\nonlin(\solA)=\int_0^\solA\frac{z^b}{(1-z)^a}\md z,
\]
we see that the equation for the biofilm density in \eqref{eq:Biofilm} is a particular case of \eqref{eq:SD-PME.Problem}.

In simulation studies, the surface on which the biofilm grows is typically the 
bottom part of the boundary, see \Cref{fig:biofilm.scheme}. This
surface is impermeable to nutrients and biomass and hence, homogeneous Neumann boundary conditions are specified for $C$ and $M$. On the lateral boundaries, homogeneous Neumann boundary conditions are assumed as well. 
Through the top boundary, nutrients are added to the system which can be modeled by Dirichlet boundary conditions, setting the level of the nutrient concentration to the concentration in the bulk liquid. For the biomass density, homogeneous Dirichlet conditions are imposed on the top boundary. 
Initially, small pockets of biomass are placed on the 
bottom boundary with initial biomass density $M_0<1$. Everywhere else in the domain $M_0$ is zero, and the substrate concentration $C_0$ is set to the value of the bulk concentration everywhere in $\Omega$.  Hence, the biofilm model is a particular case of the coupled system with mixed Dirichlet-Neumann boundary conditions \eqref{eq:SD-PME.coupled.Problem}.

The biofilm growth model \cite{Eberl2000} has been studied in a series of papers, mainly in simulation studies, and has been further extended to take additional biofilm processes into account, see e.g. \cite{EfSen},\cite{GhSoEb},\cite{RaSuEb},\cite{SoEfEb}. This led to more involved, strongly coupled systems involving several dissolved substrates and multiple types of biomass. 
The well-posedness and long-time behavior of solutions of the biofilm model \eqref{eq:Biofilm} was studied in \cite{Efendiev2009}. 
It was shown that unique global, non-negative solutions exist 
and that the system generates a semigroup that 
possesses a compact global attractor. However, the analysis was based on the assumption of homogeneous Dirichlet boundary conditions for $M$ and that the domain $\Omega$ is piece-wise smooth.

Our aim is to extend the well-posedness theory for the significantly larger class of systems \eqref{eq:SD-PME.coupled.Problem} and to allow for mixed Dirichlet-Neumann boundary data and more general domains. Having a solution theory for mixed boundary conditions is essential for biofilm modeling applications.
In \cite{Efendiev2009}, the existence of solutions  was proven by regularizing the diffusion coefficient and considering smooth, non-degenerate approximations. The solutions of the degenerate system were obtained as limits of classical solutions of the approximate problems.
The existence theory for classical solutions of quasilinear non-degenerate parabolic equations is well-established, see for example \cite{lady1968} or \cite{lieberman2005}, and relies on Schauder's fixed point theorem and uniform bounds on the H\"older norms of the solutions.
The approach in \cite{Efendiev2009} has the advantage that all computations are rigorous by the smoothness of the approximate solutions. 
Moreover, interior continuity of the solutions can be deduced from \cite{Sacks83}.
Indeed, the modulus of continuity of the approximate solutions does not depend on the regularization parameter. Therefore, the interior continuity is transferred to the solution, since it is the point-wise limit of classical solutions.
The disadvantage of the method is that it does not apply to more general domains for non-vanishing Dirichlet boundary data and for mixed Dirichlet-Neumann boundary conditions, since the theory of classical solutions requires higher regularity of the boundary.

To overcome these restrictions, we base our approach on the existence result in \cite{Alt-Luck1983}.
This influential paper uses a time-discretization scheme with Galerkin approximations and energy estimates to prove the existence of solutions for a very general class of quasilinear problems.
In particular, the theory covers Lipschitz domains and mixed boundary conditions such as \eqref{eq:SD-PME.Problem.Condition.Dirichlet}.
However, we do not obtain classical solutions that approximate the solutions of the degenerate problem and cannot conclude the continuity of solutions based on \cite{Sacks83}. 
Nevertheless, we can show the continuity of solutions using intrinsic scaling methods which is subject of a subsequent work. In particular, in \cite{HiMu} we prove the interior H\"older continuity of solutions for a class of coupled systems including the biofilm growth model \eqref{eq:Biofilm}.

Crucial for our uniqueness proof and the continuous dependence on initial data is a $L^1$-contraction result.
In \cite{Efendiev2009}, such an estimate was shown by considering a parabolic problem
for the difference of two solutions of equation \eqref{eq:SD-PME.Problem}.
The arguments require additional regularity of the solutions and only hold for solutions bounded away from $1$.
We show a more general $L^1$-contraction result by adjusting the proof in \cite{Otto95} to our setting.
The approach in \cite{Otto95} is based on the doubling of the time-variable and holds for a large class of degenerate problems.
It uses weaker notions of solutions than we consider and does not cover time-dependent reaction terms, which is important in our case. 
Hence, we show that the $L^1$-contraction result can be extended for equations of the form \eqref{eq:SD-PME.Problem} assuming a Lipschitz condition for the reaction term.
Our result also covers a comparison principle which 
we need to study the coupled system \eqref{eq:SD-PME.coupled.Problem}.
To prove the $L^1$-contraction we derive a chain rule for the time derivative in \eqref{eq:SD-PME.Problem.DiffEqn}.
This chain rule is formulated in greater generality than we actually need, but 
it might be of independent interest.
Finally, we use $L^1$-contraction and Banach's Fixed Point Theorem to show the well-posedness of the coupled system \eqref{eq:SD-PME.coupled.Problem}.

We study the scalar equation \eqref{eq:SD-PME.Problem} in a general setting.
Although our main interest are equations where the diffusion coefficient is degenerate in $0$ and $\nonlin$ is singular in $1$, our analysis also covers non-singular and/or non-degenerate equations.
Therefore, our results apply to a wide class of reaction-diffusion models.
For instance, it includes models for collective cell spreading considered in \cite{SiBaMC} or 
the Porous-Fischer equation $\solA_t=\Delta\solA^m+\solA(1-\solA)$, $m>1$, studied e.g. in \cite{mccue2019hole}.

The outline of our paper is as follows. 
In \Cref{sect:Assumptions}, we state our assumptions, introduce the class of solutions we consider and 
formulate the main results. 
Note that we first consider the scalar equation \eqref{eq:SD-PME.Problem} in the general setting, covering both the degenerate and non-degenerate case in a unified approach.
The theory for the single equation is then used to study the coupled system \eqref{eq:SD-PME.coupled.Problem}.
In \Cref{sect:Preliminaries}, we prove a chain rule for the time derivative in \eqref{eq:SD-PME.Problem.DiffEqn} and show the equivalence of two solution concepts. 
The proof of the chain rule exploits properties of Bochner spaces and Steklov averaging that 
we provide in \Cref{sect:Appendix}.
In \Cref{sect:L1-contraction}, the chain rule is used to prove the $L^1$-contraction and an energy estimate. Moreover, we show the well-posedness for the scalar equation \eqref{eq:SD-PME.Problem}.
Finally, in \Cref{sect:Well-posedness}, we use these results and Banach's Fixed Point Theorem to prove the well-posedness of the coupled system \eqref{eq:SD-PME.coupled.Problem}.

\section{Hypotheses and main results}\label{sect:Assumptions}

We consider the boundary-/initial value problem \eqref{eq:SD-PME.Problem}.
We assume that $\Omega\subset \IR^N$ is a bounded Lipschitz domain and that $\Gamma\subseteq\partial\Omega$ is measurable with Hausdorff measure $H^{N-1}(\Gamma)>0$.
Let $\Omega_T=\Omega\times (0,T]$, $T>0$ and ${I}\subseteq \IR$ be an open interval.

\begin{remark}
	The interval ${I}$ is introduced to provide a unified approach for both, degenerate and non-degenerate equations. 
	For instance, we include semilinear equations, where $\nonlin(z)=z$, by setting ${I}=\IR$ and degenerate quasilinear equations with additional singularity, where $\nonlin'(0)=0$ and $\nonlin(1)=\infty$, by setting ${I}=(-1,1)$.
\end{remark}

Our well-posedness result is based on the following assumptions.
The function $\nonlin:{I}\to\IR$ satisfies the structural assumptions:
\newcounter{counterHypotheses}
\begin{enumerate}[label=\upshape(H\arabic*),ref=\upshape H\arabic*]
	\item $\nonlin$ is continuous and strictly increasing;  \label{itm:nonlin.basic.assumption}
	\item $\nonlin$ is surjective; \label{itm:nonlin.blow-up.assumption}
	\item $\nonlin$ is piece-wise continuously differentiable and either $\nonlin'\geq \alpha>0$ for a constant $\alpha$ or there exists a $z_0\in {I}$ such that $\nonlin'(z_0)=0$ and $\nonlin$ is convex on ${I}\cap [z_0,\infty)$ and concave on ${I}\cap (-\infty,z_0]$. \label{itm:nonlin.convex.differentiable.assumption}
	\setcounter{counterHypotheses}{\value{enumi}}
\end{enumerate}
Without loss of generality, we assume that 
$0\in {I}$ and $\nonlin(0)=0$, otherwise we can replace $\nonlin$ by $\nonlin(\emptyarg+c_1)+c_2$ for some constants $c_1$ and $c_2$ and transform $\source$ accordingly.
The possible blow-up behaviour is encoded in \eqref{itm:nonlin.blow-up.assumption}, if the interval ${I}$ is finite.
For instance, if $I=(-1,1)$, then \eqref{itm:nonlin.blow-up.assumption} implies that $\nonlin(\pm1)=\pm\infty$.
Finally, we point out that the assumptions \eqref{itm:nonlin.basic.assumption} and 
\eqref{itm:nonlin.blow-up.assumption} are sufficient to prove our well-posedness results. 
The last property  \eqref{itm:nonlin.convex.differentiable.assumption} is a technical assumption that allows us to show additional regularity of the solutions.
Here, we may also assume without loss of generality that $z_0=0$, by picking $c_1=-z_0$ when replacing $\nonlin$ above.

\newcounter{counterHypothesesReactionTerms}
The reaction function $\source:\Omega_T\times I\to\IR$ is measurable and satisfies the assumptions:
\begin{enumerate}[label=\upshape(R\arabic*),ref=\upshape R\arabic*]
	\item\label{itm:source.Lipschitz.assumption} $\source$ is uniformly Lipschitz continuous with respect to the last argument, that is, there exists a $\LipBoundSource\geq 0$ such that
	\begin{equation*}
		\norm{\source(\emptyarg,z_1)-\source(\emptyarg,z_2)}_{L^\infty(\Omega_T)}\leq \LipBoundSource\abs{z_1-z_2}\quad\text{for all $z_1,z_2\in {I}$;}
	\end{equation*}
	\item\label{itm:source.growth.condition} $\source$ satisfies the growth condition
	\begin{equation*}
		\abs{\source(\emptyarg,z)}\leq C(1+\Nonlin(z)^{\frac{1}{2}})\quad \text{for all $z\in{I}$.}
	\end{equation*}
	\setcounter{counterHypothesesReactionTerms}{\value{enumi}}
\end{enumerate}
It is important to point out that \eqref{itm:source.growth.condition} follows from \eqref{itm:source.Lipschitz.assumption} if the interval ${I}$ is bounded.
Moreover, in this case, extending $\source$ to $\Omega_T\times\IR$, the assumption  \eqref{itm:source.Lipschitz.assumption} holds if $\source$ is locally Lipschitz continuous with respect to $z\in\IR$. 
However, if ${I}$ is unbounded and $\nonlin$ does not have a singularity, \eqref{itm:source.Lipschitz.assumption} is a rather restrictive assumption.

For what follows we introduce $\Nonlin(z):=\int_0^z\nonlin(\tilde{z})\md z$.
The quantity $\int_{\Omega}\Nonlin(\solA)$ is interpreted as the energy of the solution.
\begin{enumerate}[label=\upshape(BC),ref=\upshape BC]
	\item\label{itm:Dirichlet.boundary.condtion} 
	The Dirichlet data $\solA^D:\Omega\to I$ is measurable and $\nonlin(\solA^D)\in L^\infty(\Omega)\cap H^1(\Omega)$.
\end{enumerate}
\begin{enumerate}[label=\upshape(IC),ref=\upshape IC]
	\item\label{itm:initial.condtion} 
	The initial data $\solA_0:\Omega\to {I}$ is measurable and $\Nonlin(\solA_0)\in L^1(\Omega)$.
\end{enumerate}

The spatial regularity of $\solA^D$ allows us to formulate \eqref{eq:SD-PME.Problem.Condition.Dirichlet} in trace sense.
To this end we define the closed subspace 
\[
{V}=\left\{\solC\in H^1(\Omega)\mid \Tr\solC =0\ \text{a.e.\ on}\ \Gamma\right\}
\]
of $H^1(\Omega)$ and will require that solutions satisfy $\nonlin(\solA)\in \nonlin(\solA^D)+L^2(0,T;{V})$. 
The dual space of $V$ is denoted by $V^*$.
Observe that elements of ${V}$ satisfy the Poincar\'e inequality, since $\Gamma$ has strictly positive measure, see Theorem 7.91 
in \cite{Salsa2016}.

\begin{remark}
	From \eqref{itm:nonlin.basic.assumption}, \eqref{itm:Dirichlet.boundary.condtion} and \eqref{itm:initial.condtion} it is inferred that $\solA_0\in L^1(\Omega)$.
	Indeed, \eqref{itm:nonlin.basic.assumption} and the assumption $\nonlin(0)=0$ imply that $\Nonlin(z_1)-\Nonlin(z_2)\geq \nonlin(z_2)(z_1-z_2)$ for all $z_1,z_2\in{I}$.
	Let $z\in{I}$, $\delta>0$ and set $z_1={z}$, $z_2=\mathrm{sign}(z)\beta(\frac{1}{\delta})$, where $\beta:=\nonlin^{-1}$.
	We conclude that $\Nonlin(z)\geq\Nonlin(z)-\Nonlin(z_2)\geq\frac{1}{\delta}\mathrm{sign}(z)(z-z_2) = \frac{1}{\delta}\left(\abs{z}-\beta(\frac{1}{\delta})\right)$.
	Therefore,
	\begin{equation}\label{eq:estimate.by.energy.functional}
		\abs{z}\leq \delta\left(\Nonlin(z)+\beta\left(\frac{1}{\delta} \right)\right),\quad z\in {I},\ \delta>0.
	\end{equation}
	In particular, $\abs{\solA_0}\leq c_1\Nonlin(\solA_0) +c_2$ for certain $c_1,c_2>0$, so $\solA_0\in L^1(\Omega)$.
	
	In fact, most of the analysis only requires that $\solA_0\in L^1(\Omega)$ and $\Nonlin(\solA_0;\solA^D)\in L^1(\Omega)$, where
	\[
	\Nonlin(z;\bar{z}):=\int_{\bar{z}}^z\nonlin(\tilde{z})-\nonlin(\bar{z})\ \md \tilde{z},
	\]
	and $\int_\Omega \Nonlin(\solA;\bar{\solA})$ is the relative energy functional of $\solA$ relative to $\bar{\solA}$.
	This is important to remark in the case of more general boundary data $u^D$. 
	The assumption $\Nonlin(\solA_0;\solA^D)\in L^1(\Omega)$ implies that $\solA_0$ has finite energy relative to $\solA^D$.
	However, since $\nonlin(\solA^D)$ is bounded by \eqref{itm:Dirichlet.boundary.condtion}, in our setting, it is sufficient to impose $\Nonlin(\solA_0)\in L^1(\Omega)$, i.e.\ $\solA_0$ has finite (absolute) energy.
\end{remark}

We consider the following class of solutions. 
Here, we denote by $\pinprod{\emptyarg}{\emptyarg}$ the $\pinprod{{V}^*}{V}$ dual pairing and by $\inprod{\emptyarg}{\emptyarg}$ the inner product in $L^2(\Omega)$. 
Further, we define the function spaces
\begin{align*}
	\mcV= L^2(0,T;{V})\quad\text{and}\quad \mcW=\{\solA\in L^\infty(0,T;L^1(\Omega))\mid \solA_t\in L^2(0,T;{V}^*)\},
\end{align*}
where $\solA_t$ is understood in the sense of distributions. Therefore, $\solA\in \mcW$ if and only if $\solA\in L^\infty(0,T;L^1(\Omega))$ and there exists a (necessarily unique) $\solB\in L^2(0,T;{V}^*)$ such that
\begin{equation}\label{eq:solution.space.time.derivative}
	\int_0^T\pinprod{\solB}{\eta}=-\int_0^T\solA\eta_t
\end{equation}
for all $\eta\in W^{1,1}(0,T;L^\infty(\Omega))\cap L^2(0,T;{V})$ such that $\eta(0)=\eta(T)=0$.
We write $\solA_t=\solB$.

\begin{definition}\label{def:SD-PME.solution.including.time-derivative}
	A measurable function $\solA:\Omega_T\rightarrow I$ is a \emph{solution} of \eqref{eq:SD-PME.Problem} if
	\begin{enumerate}[label=(\roman*),font=\itshape]
		\item $\solA\in \mcW$, $\nonlin(\solA)\in \nonlin(\solA^D)+\mcV$,
		\item $\solA$ satisfies the initial condition, i.e.\ the identity
		\begin{equation}\label{eq:SD-PME.solution.including.time-derivative.init.cond.id}
			\displayindent0pt
			\displaywidth\textwidth
			\int_0^T\left(\pinprod{\solA_t}{\testA}+\int_{\Omega}(\solA-\solA_0)\testA_t\right)=0
		\end{equation}
		holds for all $\eta\in \mcV \cap W^{1,1}(0,T;L^\infty(\Omega))$ with $\eta(T)=0$, and
		\item $\solA$ satisfies \eqref{eq:SD-PME.Problem.DiffEqn} in the distributional sense, i.e.\ the identity
		\begin{equation}\label{eq:SD-PME.solution.including.time-derivative.id}
			\displayindent0pt
			\displaywidth\textwidth
			\begin{aligned}
				\int_0^T\left[\pinprod{\solA_t}{\testA}+\inprod{\nabla\nonlin(\solA)}{\nabla\testA}\right]=\int_0^T\inprod{\source(\emptyarg,\solA)}{\eta}
			\end{aligned}
		\end{equation}
		holds for all $\eta\in \mcV$.
	\end{enumerate}
\end{definition}

We remark that this definition can be simplified if ${I}$ is bounded or if $\beta:=\nonlin^{-1}$ is Lipschitz continuous and $\solA_0\in {V}^*$.
In both cases, $\solA$ can be interpreted as an element in $L^2(0,T;{V}^*)$.
Then, \eqref{eq:SD-PME.solution.including.time-derivative.init.cond.id} implies that $\solA\in H^1(0,T;{V}^*)$ with weak derivative $\solA_t$ and $\solA(0)=\solA_0\in {V}^*$.
\begin{remark}
	We will show that \Cref{def:SD-PME.sub-super.solution.including.time-derivative} \textit{(i)} implies that $\Phi(\solA)\in L^\infty(0,T;L^1(\Omega))$, see \Cref{cor:time-regularity.energy.functional}.
	Then, \eqref{itm:source.growth.condition} implies that $\source(\emptyarg\solA)$ is in $L^\infty(0,T;L^2(\Omega))$ so that the right-hand side of \eqref{eq:SD-PME.solution.including.time-derivative.id} is well-defined.
\end{remark}

\begin{definition}\label{def:SD-PME.sub-super.solution.including.time-derivative}
	A measurable function $\solA:\Omega_T\rightarrow {I}$ is a \emph{subsolution} (\emph{supersolution}) of \eqref{eq:SD-PME.Problem} if
	\begin{enumerate}[label=(\roman*),font=\itshape]
		\item $\solA\in \mcW$ and $\nonlin(\solA)\in \mcV$ with $\Tr\nonlin(\solA)-\nonlin(\solA^D)\leq0\ (\geq 0)$ a.e.\ in $\Gamma\times(0,T)$,  
		\item  identity \eqref{eq:SD-PME.solution.including.time-derivative.init.cond.id} holds with $=$ replaced by $\geq(\leq)$ for all test functions with $\testA(0)\geq 0$, and
		\item  identity \eqref{eq:SD-PME.solution.including.time-derivative.id} holds with $=$ replaced by $\leq(\geq)$ for all non-negative test functions $\testA$.
	\end{enumerate}
\end{definition}

The following two theorems imply the well-posedness of \eqref{eq:SD-PME.Problem}.
Observe that \Cref{thm:L1-contraction} also covers a comparison principle.
We use the notation $a_+=\max\{a,0\}$ and $a_-=(-a)_+$.

\begin{theorem}[$L^1$-contraction]\label{thm:L1-contraction}
	Suppose that $\nonlin$ satisfies \eqref{itm:nonlin.basic.assumption}.
	Let $\solA$ and $\solAvar$ be solutions of \eqref{eq:SD-PME.Problem} with reaction functions $\source$ and $\sourcevar$ and initial data $\solA_0$ and $\solAvar_0$, respectively, where $\source$ and $\sourcevar$ satisfy \eqref{itm:source.Lipschitz.assumption}
	and $\solA_0$ and $\solAvar_0$ satisfy \eqref{itm:initial.condtion}.
	Furthermore, we assume that both solutions satisfy the same Dirichlet data $\solA^D$ and \eqref{itm:Dirichlet.boundary.condtion} holds.
	Then,
	\begin{equation}\label{eq:L1-contraction}
		\norm{\solA(t)-\solAvar(t)}_{L^1(\Omega)}\leq e^{\LipBoundSource t}\left(\norm{\solA_0-\solAvar_0}_{L^1(\Omega)}+\int_0^t\norm{\source(\emptyarg,\solA)-\sourcevar(\emptyarg,\solA)}_{L^1(\Omega)}\right)
	\end{equation}
	for all $t\in [0,T]$, where $\LipBoundSource$ is the Lipschitz constant in \eqref{itm:source.Lipschitz.assumption}.
	
	Moreover, if $\solA$ is a subsolution and $\solAvar$  a supersolution, then 
	\begin{equation}\label{eq:L1-contraction.sub-supersolution}
		\norm{(\solA(t)-\solAvar(t))_+}_{L^1(\Omega)}\leq e^{\LipBoundSource t}\left(\norm{(\solA_0-\solAvar_0)_+}_{L^1(\Omega)}+\int_0^t\norm{\source(\emptyarg,\solB)-\sourcevar(\emptyarg,\solB)}_{L^1(\Omega)}\right)
	\end{equation}
	for almost every $t\in(0,T)$, where $\solB$ can be either $\solA$ or $\solAvar$.
\end{theorem}

\begin{theorem}[Well-posedness]\label{thm:Well-posedness}
	Suppose that \eqref{itm:nonlin.basic.assumption}, \eqref{itm:nonlin.blow-up.assumption},
	\eqref{itm:source.growth.condition},
	\eqref{itm:Dirichlet.boundary.condtion} and \eqref{itm:initial.condtion} hold. 
	Then, there exists a solution $\solA$ of \eqref{eq:SD-PME.Problem} 
	and $\solA$ satisfies the energy estimate
	\begin{equation}\label{eq:energy_estimate}
		\begin{aligned}
			&\ \norm{\Nonlin(\solA;\solA^D)}_{L^\infty(0,T;L^1(\Omega))}+\norm{\nabla\nonlin(\solA)}_{L^2(\Omega_T)}^2\\
			\leq&\ C \Bigl(\norm{\Nonlin(\solA_0;\solA^D)}_{L^1(\Omega)}+\norm{\nabla\nonlin(\solA^D)}_{L^2(\Omega_T)}+\norm{\source(\emptyarg,\solA)}_{L^2(\Omega_T)}\norm{\nonlin(\solA)-\nonlin(\solA^D)}_{L^2(\Omega_T)}\Bigr),
		\end{aligned}
	\end{equation}
	for some constant $C\geq 0$.
	In addition, suppose that \eqref{itm:source.Lipschitz.assumption} holds.
	Then, $\solA$ is the unique solution of \eqref{eq:SD-PME.Problem} and it depends continuously on the initial data $\solA_0$ via the $L^1$-contraction estimate \eqref{eq:L1-contraction}.
	
	Furthermore, suppose that $\source$ is also bounded.
	If $\nonlin(\solA_0)\in L^\infty(\Omega)$, then $\nonlin(\solA)$ is bounded by a constant depending on $\Omega$ and on the bounds of $\source$ and $\nonlin(\solA_0)$.
	If \eqref{itm:nonlin.convex.differentiable.assumption} holds, then $\solA\in C([0,T];L^1(\Omega))$.
	Finally, if both assumptions hold and $\nonlin(\solA_0)\in \nonlin(\solA^D)+{V}$, then we have $\nonlin(\solA)\in H^1(0,T;L^2(\Omega))$.
\end{theorem}

\begin{remark}[Energy estimate]\label{rem:energy_uniform_bound}
	We formulate estimate \eqref{eq:energy_estimate} in terms of the relative energy.
	In \cite{Alt-Luck1983}, an energy estimate for the absolute energy of the form
	\begin{equation}\label{eq:energy_uniform_bound}
		\norm{\Nonlin(\solA)}_{L^\infty(0,T;L^1(\Omega))}+\norm{\nabla\nonlin(\solA)}_{L^2(\Omega_T)}^2 \leq C,
	\end{equation}
	for some $C\geq 0,$ is established and it is a key ingredient in their analysis. 
	The estimate \eqref{eq:energy_uniform_bound} can be obtained from \eqref{eq:energy_estimate}.
	Indeed, by observing that
	\[
	\Nonlin(\solA;\solA^D)=\Nonlin(\solA)-\Nonlin(\solA^D)-\nonlin(\solA^D)(\solA-\solA^D)
	\]
	and using \eqref{eq:estimate.by.energy.functional} with $\delta=\frac{1}{2}\norm{\nonlin(\solA^D)}_{L^\infty(\Omega)}^{-1}$ to estimate $\abs{\nonlin(\solA^D)}\abs{\solA}\leq \frac{1}{2}\Nonlin(\solA)+C$ we find the inequality
	\[
	\Nonlin(\solA)\leq C\left(\Nonlin(\solA;\solA^D)+1\right).
	\]
	The reaction term in \eqref{eq:energy_estimate} can be treated in a standard way, i.e.\ we use Young's inequality and Poincar\'e's inequality to estimate the term by
	\[
	C\left(\norm{\source(\emptyarg,\solA)}_{L^2(\Omega_T)}^2+1\right)+\frac{1}{2}\norm{\nabla\nonlin(\solA)}^2_{L^2(\Omega)}.
	\]
	The second term is absorbed in the left-hand side of \eqref{eq:energy_estimate} and \eqref{itm:source.growth.condition} is applied to the first term.
	In this way we infer from \eqref{eq:energy_estimate} an estimate of the form
	\[
	\norm{\Nonlin(\solA(T))}_{L^1(\Omega)}+\norm{\nabla\nonlin(\solA)}_{L^2(\Omega_T)}^2\leq C\left(\int_0^T\norm{\Nonlin(\solA(t))}_{L^1(\Omega)}\md t+1\right).
	\]
	An application of Gronwall's Lemma shows that \eqref{eq:energy_uniform_bound} holds.
\end{remark}

Theorems \ref{thm:L1-contraction} and \ref{thm:Well-posedness} provide the basis for proving the well-posedness of the coupled system \eqref{eq:SD-PME.coupled.Problem}.
Let $\Gamma_1,\Gamma_2\subseteq\partial\Omega$ be measurable with 
Hausdorff measure $H^{N-1}(\Gamma_i)>0$, $i=1,2.$
In this case, for simplicity and motivated by applications, where solutions describe densities or concentrations,
we assume that $I=[0,1)$. 
For the reaction functions we assume the following. 

\begin{enumerate}[label=\upshape(R\arabic*),ref=\upshape R\arabic*]
	\setcounter{enumi}{\value{counterHypothesesReactionTerms}}
	\item\label{itm:source.Lipschitz.assumption.coupled} 
	The functions ${f},{g}:\Omega_T\times [0,1)\times[0,1]\to\IR$ are measurable and uniformly Lipschitz continuous with respect to the last argument, that is, there exists a $\LipBoundSource\geq 0$ such that
	\begin{align*}
		&\norm{{f}(\emptyarg,u_1,v_1)-{f}(\emptyarg,u_2,v_2)}_{L^\infty(\Omega_T)}+\norm{{g}(\emptyarg,u_1,v_1)-{g}(\emptyarg,u_2,v_2)}_{L^\infty(\Omega_T)}\\
		\qquad& \leq \LipBoundSource(\abs{u_1-u_2}+\abs{v_1-v_2})
	\end{align*}
	for all $u_1,u_2\in [0,1)$ and $v_1,v_2\in[0,1]$.
	\item\label{itm:source.positivity.coupled} Moreover, we assume that ${f}(\emptyarg,0,\solB)\geq 0$, ${g}(\emptyarg,\solA,0)\geq 0$ and ${g}(\emptyarg,\solA,1)\leq 1$ a.e.\ in $\Omega_T$, for all $\solA\in[0,1)$ and $\solB\in[0,1]$.
\end{enumerate}
Condition \eqref{itm:source.positivity.coupled} ensures that solutions remain non-negative and that $v$ is bounded by $1$.
The non-negativity of the solutions also validates the choice of the interval ${I}=[0,1)$, which is not an open interval of $\IR$ as assumed previously.
However, we can simply extent the structural functions to the open interval $(-1,1)$ in an appropriate manner and apply the well-posedness theory for \eqref{eq:SD-PME.Problem}, which then yields solutions taking values in $[0,1)$.

As above, we introduce the spaces
\[
{V}_i:=\left\{\solC\in H^1(\Omega)\mid \Tr\solC =0\ \text{a.e.\ on}\ \Gamma_i\right\},\quad i=1,2,
\]
and define $\mcV_i$ and $\mcW_i$, $i=1,2$, in an obvious manner. Solutions of system \eqref{eq:SD-PME.coupled.Problem} are analogously defined as in \Cref{def:SD-PME.solution.including.time-derivative} for problem \eqref{eq:SD-PME.Problem}.

\begin{theorem}[Well-posedness of the coupled system]\label{thm:Well-posedness.biofilm}
	Suppose that $\nonlin:[0,1)\to[0,\infty)$ satisfies \eqref{itm:nonlin.basic.assumption} and \eqref{itm:nonlin.blow-up.assumption} and ${f},{g}:\Omega_T\times[0,1)\times[0,1]\to \IR$ satisfy \eqref{itm:source.Lipschitz.assumption.coupled} and \eqref{itm:source.positivity.coupled}. Moreover, the boundary data $\solA^D:\Omega\to[0,1)$ satisfies \eqref{itm:Dirichlet.boundary.condtion}, $\solB^D:\Omega\to[0,1]$ satisfies $\solB^D\in H^1(\Omega)$ and the initial data  $\solA_0$ satisfies \eqref{itm:initial.condtion} and $\solB_0:\Omega\to[0,1]$ is measurable.
	Then, there exists a unique solution $(\solA,\solB):\Omega\to [0,1)\times[0,1]$ of \eqref{eq:SD-PME.coupled.Problem} in $L^\infty(0,T;L^1(\Omega))\times C([0,T];L^1(\Omega))$, the solution satisfies 
	a $L^1$-contraction estimate and an energy inequality analogous to \eqref{eq:L1-contraction} and \eqref{eq:energy_estimate}, respectively.
	Furthermore, if $\nonlin$ satisfies \eqref{itm:nonlin.convex.differentiable.assumption}, then $(\solA,\solB)\in C([0,T];L^1(\Omega)\times L^1(\Omega))$.
	
	Finally, if $\solA_0\leq 1-\theta$ for some $\theta\in(0,1)$, then $\solA\leq 1-\mu$ for some $\mu\in(0,1)$ depending on $\Omega$, $\LipBoundSource$ and $\nonlin(1-\theta)$.
	In this case, if $\nonlin(\solA_0)\in \nonlin(\solA^D)+{V}$ and $\nonlin$ satisfies \eqref{itm:nonlin.convex.differentiable.assumption}, then $\nonlin(\solA)\in H^1(0,T;L^2(\Omega))$.
\end{theorem}

The last property in Theorem \ref{thm:Well-posedness.biofilm} is crucial in modeling applications. 
In fact, the singularity in the primitive of the diffusion coefficient ensures that the biomass density remains bounded by a constant strictly less than one.
Furthermore, we provided the additional regularity $\solA\in C([0,T];L^1(\Omega))$ so that our well-posedness theory agrees with the results in \cite{Efendiev2009}.
In this reference the continuity is needed to apply the theory of infinite dimensional dynamical systems.

\section{Preliminaries}\label{sect:Preliminaries}

The aim of this section is to derive a chain rule for the time derivative of the scalar equation \eqref{eq:SD-PME.Problem} and to prove some related properties.
To this end we introduce Steklov averages and an appropriate functional setting.
The functional spaces are necessary to cover the equation in full generality.
For instance, suppose we restrict the problem to the biofilm model \eqref{eq:Biofilm} and consider 
$\nonlin:{I}\to\IR$ with a bounded interval ${I}$.
Then, $\solA_0$ and $\solA$ are bounded functions and the space $\mcW$ can be replaced by $H^1(0,T;{V}^*)$, since $\solA_0\in L^2(\Omega)\subseteq {V}^*$.
This simplifies the theory significantly.
A similar argument can be made in the semi-linear case.
However, in the general case, $\solA_0\in L^1(\Omega)$, $\solA\in L^\infty(0,T;L^1(\Omega))$ and the space $L^1(\Omega)$ is not related to ${V}^*$ in a natural manner and hence, 
introducing the space $\mcW$ is necessary.

\subsection*{Analysis on Bochner spaces and Steklov averaging}

\begin{definition}[Steklov averaging]\label{def:Steklov.average}
	Let $X$ be a Banach space, $\solA\in L^1_\loc(\IR;X)$ and $h>0$.
	We define the \emph{(backward) Steklov average} $\solA^h:\IR\to X$ by
	\[
	\solA^h(t):=\frac{1}{h}\int_{t-h}^t \solA(s)\md s
	\]
	for all $t\in\IR$.
\end{definition}

\begin{remark}\label{rem:Steklov.average.extent.functions}
	We interpret $L^1(0,T;X)\subset L^1_\loc(\IR;X)$ by setting $\solA\equiv0$ on 
	$\mathbb{R}\setminus[0,T]$.
\end{remark}

Several properties of Bochner spaces and Steklov averages that we will use in the sequel are shown in Appendix \ref{sect:Appendix}.

\begin{remark}\label{rem:embedded.spaces.and.Bochner.and.Steklov}
	Suppose $X$ and $Y$ are Banach spaces such that $X$ is continuously embedded in $Y$.
	Applying Hille's Theorem, see Theorem 1.2.4 
	in \cite{Hyt-Neerv-Ver-Weis2016}, to the inclusion map shows that the Bochner integrals with respect to $X$ and with respect to $Y$ coincide for $X$-valued functions.
	In particular, if $\solA \in L^1_\loc(\IR;X)\cap L^1_\loc(\IR;Y)$, then $\solA^h$ is unambiguously defined.
\end{remark}

\begin{remark}[Point-wise value of the Steklov average]\label{rem:Steklov.point-wise.value}
	If the embedding $X\hookrightarrow L^1(\Omega)$ is continuous, then \Cref{rem:embedded.spaces.and.Bochner.and.Steklov} implies that, for any $\solA\in L^1(\Omega_T)$,
	\[
	\solA^h(x,t)=\frac{1}{h}\int_{t-h}^t\solA(x,s)\md s
	\]
	for all $t\in[0,T]$ and almost every $x\in\Omega$.
	Indeed, using Fubini's Theorem one can check that $L^1(0,T;L^1(\Omega))\cong L^1(\Omega_T)$, where the isometric isomorphism is given by $\solA(x,t)=[\solA(t)](x)$.
	Then, the function $\solB\in L^1(\Omega)$ defined by $\solB=\int_0^T\solA(t)\md t$, for some $\solA\in L^1(\Omega_T)$, satisfies $\int_\Omega\solB=\iint_{\Omega_T}\solA$ and $\solB(x)=\int_0^T\solA(x,t)\md t$ for almost every $x\in\Omega$.
	The identity then follows as a special case.
\end{remark}

In view of \Cref{def:SD-PME.solution.including.time-derivative}, we review some properties of the relevant functional spaces.
Write $Y:=L^\infty(\Omega)\cap {V}$ and note that $L^1(\Omega)\hookrightarrow L^\infty(\Omega)^*\subset Y^*$ and $V^*\subset Y^*$.
This structure allows us to define $L^1(\Omega)\cap {V}^*$ and we have the following commutative diagram 
\[
\begin{tikzcd}
L^1(\Omega)\arrow[r, hook] 	& L^\infty(\Omega)^* \arrow[r, hook] &Y^*\\
L^1(\Omega)\cap {V}^* \arrow[u, hook] \arrow[rr, hook]& 		&{V}^*.\arrow[u, hook]
\end{tikzcd}
\]
The continuous embedding into ${Y}^*$ implies that a Cauchy sequence in $L^1(\Omega)\cap {V}^*$ has the same limit with respect to the topologies of both, $L^1(\Omega)$ and ${V}^*$.
We conclude that
the normed vector space $L^1(\Omega)\cap{V}^*$ is complete and for any $\solA\in L^1(\Omega)\cap{V}^*$ we have 
\begin{equation}\label{eq:interplay.L1-V*}
\pinprod{\solA}{\testA}=\int_{\Omega}\solA\testA\quad\text{for all}\ \testA\in L^\infty(\Omega)\cap {V}.
\end{equation}

Either expression in \eqref{eq:interplay.L1-V*} describes the $(Y^*,Y)$ pairing and therefore, we denote the $(Y^*,Y)$ pairing by $\pinprod{\emptyarg}{\emptyarg}$ as well.
Let $\solA\in\mcW$ and observe that $\solA,\solA_t\in L^2(0,T;Y^*)$.
From \eqref{eq:solution.space.time.derivative} it follows that 
\[
\int_0^T\testB\solA_t=-\int_0^T\testB_t\solA\quad\text{in}\ Y^*\ \text{for all $\testB\in C^\infty_c((0,T))$.}
\]
Indeed, let $\testC\in Y$ and consider $\testA(x,t)=\testB(t)\testC(x)$, where $x\in\Omega$, $t\in[0,T]$, to conclude the identity.
Therefore, $\solA_t$ is the weak derivative of $\solA$ as $Y^*$-valued functions.
It follows that $\mcW\subseteq W^{1,2}(0,T;Y^*)$.
In particular, $\solA\in C([0,T];Y^*)$ and therefore $\solA(0)\in {Y}^*$ is well-defined.

In the well-posedness proof we will use the following 
equivalent formulation for the initial data.  
\begin{lemma}
	\label{lem:initial.cond.alternative}
	Suppose $\solA\in \mcW$ and $\solA_0\in L^1(\Omega)$, then \eqref{eq:SD-PME.solution.including.time-derivative.init.cond.id} holds if and only if
	\begin{equation}\label{eq:initial.cond.alternative}
	\pinprod{\solA(0)}{\testA}=\int_{\Omega}\solA_0\testA\quad\text{for all}\ \testA\in L^\infty(\Omega)\cap {V}.
	\end{equation}
	The statement still holds if we replace $=$ by $\geq$ ($\leq$) in \eqref{eq:SD-PME.solution.including.time-derivative.init.cond.id} for $\testA(0)\geq 0$ and $=$ by $\leq$ ($\geq$) in \eqref{eq:initial.cond.alternative} with $\testA\geq 0$.
\end{lemma}

\begin{proof}
	Observe that $\solA(t)-\solA(0)=\int_0^t\solA_t(s)\md s$ in $Y^*$ for all $t\in[0,T]$ by \Cref{prop:Fund.Thm.Cal.Bochner_spaces}. 
	Moreover, for any $\testA\in W^{1,2}([0,T];Y)$ we have that $\pinprod{\solA}{\testA}\in W^{1,2}(0,T)$ and its derivative is given by $\pinprod{\solA}{\testA_t}+\pinprod{\solA_t}{\testA}$.
	Indeed, the mapping $t\mapsto \pinprod{\solA(t)}{\testA(t)}$ is absolutely continuous and its a.e.-derivative is given by the product rule.
	
	Suppose \eqref{eq:SD-PME.solution.including.time-derivative.init.cond.id} holds with $\geq$ instead of $=$.
	Fix $\testC\in Y$, $\testC\geq 0$ and set $\testA(x,t)=(T-t)\testC(x)$, then
	\begin{align*}
	\pinprod{\solA(0)}{\testC}&=-\frac{1}{T}\int_0^T\frac{\md }{\md t}\pinprod{\solA}{\testA}=\frac{1}{T}\left(\iint_{\Omega_T}\solA\testC-\int_{0}^{T}\pinprod{\solA_t}{\testA}\right)\\
	&\!\stackrel{\eqref{eq:SD-PME.solution.including.time-derivative.init.cond.id}}{\leq}\frac{1}{T}\left(\iint_{\Omega_T}\solA\testC-\iint_{\Omega_T}(\solA-\solA_0)\testC\right)=\int_{\Omega}\solA_0\testC.
	\end{align*}
	Conversely, suppose \eqref{eq:initial.cond.alternative} holds with $\leq$ instead of $=$.
	Then,
	\begin{align*}
	\int_0^T\left(\pinprod{\solA_t}{\testA}+\int_{\Omega}(\solA-\solA_0)\testA_t\right)&=\int_0^T\left(\pinprod{\solA_t}{\testA}+\int_{\Omega}\solA\testA_t\right)-\int_{0}^T\pinprod{\solA_0}{\testA_t}\\
	&\geq\int_0^T\frac{\md}{\md t}\pinprod{\solA}{\testA}+\int_{\Omega}\solA(0)\eta(0)=0
	\end{align*}
	for any $\testA\in C^\infty([0,T];Y)$, $\testA\geq 0$ with $\testA(T)=0$.
	We use mollifiers for Bochner spaces to generalize the identity to any $\testA\in W^{1,1}(0,T;L^\infty(\Omega))\cap L^2(0,T;{V})$, $\testA\geq 0$ with $\testA(T)=0$.
\end{proof}

\begin{lemma}[Steklov average of the solution]\label{lem:Steklov.averaged.solution.convergence}
	Let $\solA\in \mcW$, then $\solA^h\in W^{1,\infty}(0,T;L^1(\Omega))$ and its weak derivative is given by	
	\begin{equation}\label{eq:Steklov.averaged.time-derivative.solution}
	\partial_t\solA^h(x,t)=\frac{1}{h}\left(\solA(x,t)-\solA(x,t-h)\right)
	\end{equation}
	for almost every $(x,t)\in \Omega_T$ and any $h>0$.
	Moreover, $\partial_t\solA^h$ can be extended to an element in $L^2(0,T;{V}^*)$ such that $\partial_t\solA^h=(\solA_t)^h$ and $\solA_t^h\to\solA_t$ in $L^2(0,T;{V}^*)$ as $h\to 0$.
	The result still holds if we extend $\solA(t)=\solA_0$ for $t<0$ for some $\solA_0\in L^1(\Omega)$ satisfying \eqref{eq:SD-PME.solution.including.time-derivative.init.cond.id}.
	In this case, $\solA_t(t)$ vanishes for $t<0$.
\end{lemma}

\begin{proof}
	By assumption $\solA\in L^\infty(0,T;L^1(\Omega))\cap W^{1,2}(0,T;{Y}^*)$ and therefore, $\solA^h$ is defined unambiguously, see \Cref{rem:embedded.spaces.and.Bochner.and.Steklov}.
	From \Cref{lem:Steklov.properties} it follows that $\solA^h\in W^{1,\infty}(0,T;L^1(\Omega))\cap W^{1,2}(0,T;Y^*)$ with weak derivative $(\solA^h)_t=\solA_t^h\in L^2(0,T;{V}^*)$ as desired.
	Also, \eqref{eq:Steklov.averaged.time-derivative.solution} is implied by \Cref{lem:Steklov.properties}.
	For the final statement we extend $\solA(t)=\solA_0$ for $t<0$ instead of using \Cref{rem:Steklov.average.extent.functions}.
	By \Cref{lem:initial.cond.alternative} we see that $\solA\in L^\infty(-\infty,T;L^1(\Omega))\cap W^{1,2}_\loc(-\infty,T;{Y}^*)$ and therefore, the previous arguments are valid.
	In particular, the weak derivative of $Y^*$-valued function $\solA$ vanishes for $t<0$ and hence, $\solA_t\in L^2(-\infty,T;{V}^*)$.
\end{proof}

\subsection*{Chain rule for the time derivative}

We derive a chain rule for the time derivative in \eqref{eq:SD-PME.Problem.DiffEqn}.
The idea is to multiply the equation by $\psi(\nonlin(\solA))$ with a 
suitable function $\psi$. 
To formalize the approach, we introduce a transform that  was used in \cite{Otto95} to prove a chain rule for a weaker class of solutions and adopt the technique to our setting. 
Here, we always assume that $\phi:I\to\mathbb{R}$ satisfies \eqref{itm:nonlin.basic.assumption} and that $\solA^D$ satisfies \eqref{itm:Dirichlet.boundary.condtion}.

\begin{definition}
	Let $\Psi:\IR\to \IR$ be locally Lipschitz continuous and $\bar{z}\in{I}$. 
	We write $\psi:=\Psi'$ and $\bar{\zeta}:=\nonlin(\bar{z})$ and
	define the transformed function $\Psi^\star(\emptyarg;\bar{z}):{I}\to\IR$ by
	\begin{equation}\label{eq:transformed.function.definition}
	\Psi^\star(z;\bar{z})=\int_{\bar{z}}^z\psi(\nonlin(\tilde{z})-\bar{\zeta})\md\tilde{z}, \qquad z\in{I}.
	\end{equation}
\end{definition}

\begin{lemma}[Properties of the transform]
	Let $\bar{z}\in{I}$, then $\Psi^\star(\emptyarg;\bar{z}):{I}\to \IR$ is 
	locally Lipschitz continuous with
	\begin{equation}\label{eq:transformed.function.derivative}
	\left[\Psi^\star(\emptyarg;\bar{z})\right]'=\psi(\nonlin-\bar{\zeta}).
	\end{equation}
	Furthermore, if $\Psi$ is convex, then $\Psi^\star(\emptyarg;\bar{z})$ is convex and the inequalities
	\begin{equation}\label{eq:transformed.function.estimates}
	\psi(\nonlin(z_2)-\bar{\zeta})(z_1-z_2)\leq\Psi^\star(z_1;\bar{z})-\Psi^\star(z_2;\bar{z})\leq \psi(\nonlin(z_1)-\bar{\zeta})(z_1-z_2)
	\end{equation} 
	hold for all $z_1,z_2\in {I}$.
	In particular, if $\psi(0)=0$, then, 
	\begin{equation}\label{eq:transformed.function.positive.estimates}
	0\leq\Psi^\star(z;\bar{z})\leq \psi(\nonlin(z)-\bar{\zeta})(z-\bar{z})
	\end{equation}
	for all $z\in{I}$.
\end{lemma}

\begin{proof}
	Differentiating \eqref{eq:transformed.function.definition} with respect to $z$ yields \eqref{eq:transformed.function.derivative}.
	Next, observe that
	\begin{align*}
	\Psi^\star(z_1;\bar{z})-\Psi^\star(z_2;\bar{z})&=\int_{z_2}^{z_1}\psi(\nonlin(\tilde{z})-\bar{\zeta})\md\tilde{z}.
	\end{align*}
	If $\Psi$ is convex, then $\psi$ is non-decreasing and, if $z_1\geq z_2$, then $\psi(\phi(z_1)-\bar{\zeta})\geq \psi(\phi(z)-\bar{\zeta})\geq \psi(\phi(z_2)-\bar{\zeta})$ for all $z\in [z_2,z_1],$ where we used that $\phi$ is increasing. Hence, it follows that  
	\begin{equation*}
	\psi(\nonlin(z_2)-\bar{\zeta})(z_1-z_2)\leq\int_{z_2}^{z_1}\psi(\nonlin(\tilde{z})-\bar{\zeta})\md\tilde{z}\leq\psi(\nonlin(z_1)-\bar{\zeta})(z_1-z_2).
	\end{equation*}
	Similarly, if $z_1\leq z_2$, then $\psi(\phi(z_1)-\bar{\zeta})\leq \psi(\phi(z)-\bar{\zeta})\leq \psi(\phi(z_2)-\bar{\zeta})$ for all $z\in [z_1,z_2]$, which implies that  
	\begin{equation*}
	\psi(\nonlin(z_1)-\bar{\zeta})(z_2-z_1)\leq\int_{z_1}^{z_2}\psi(\nonlin(\tilde{z})-\bar{\zeta})\md\tilde{z}\leq\psi(\nonlin(z_2)-\bar{\zeta})(z_2-z_1).
	\end{equation*}
	Multiplying  the inequality by $-1$ we obtain  \eqref{eq:transformed.function.estimates}.
	Finally, \eqref{eq:transformed.function.positive.estimates} follows from \eqref{eq:transformed.function.estimates}
	by taking $z_1=z$ and $z_2=\bar{z}$.
\end{proof}

For instance, suppose that $\Psi(\zeta)=\frac{1}{2}\zeta^2$, then $\psi(\zeta)=\zeta$.
Setting $\bar{z}=0$, assuming that $\Phi(0)=0$ and writing $\Psi^\star=\Psi^\star(\emptyarg;0)$ 
we obtain
\[
\Psi^\star(z)=\Phi(z):=\int_0^z\nonlin(\tilde{z})\md \tilde{z}.
\]
Moreover, if \eqref{itm:nonlin.blow-up.assumption} holds, $\nonlin\in C^1(I)$ and $\beta:=\nonlin^{-1}$, then $\beta'(\zeta)=\frac{1}{\nonlin'(\beta(\zeta))}$ and we can write
\begin{equation}\label{eq:transformed.function.definition.alternative}
\Psi^\star(z)=\int_{0}^{\phi(z)} \psi(\zeta)\beta'(\zeta)\md \zeta
=\int_{0}^{\phi(z)}\zeta\beta'(\zeta)\md \zeta. 
\end{equation}
Integration by parts then implies that
\[
\Psi^\star(z)=\nonlin(z)z-B(\nonlin(z)),
\]
where $B(y):=\int_0^y\beta(s)ds$.
Actually, $B$ is a convex function, so the function $\zeta\mapsto z\zeta - B(\zeta)$ is concave.
Suppose $\bar{\zeta}$ is a critical point of this function, i.e.\ $z-B'(\bar{\zeta})=0$, then $\bar \zeta$ is a maximum and
$\beta(\bar{\zeta})=z$, so $\bar{\zeta}=\nonlin(z)$.
Therefore, we have
\[
\Phi(z)=\Psi^\star(z)=\sup_{\zeta\in\IR} \left\{z\zeta-B(\zeta)\right\}=:B^*(z),
\]
where $B^*$ denotes the usual Legendre transform of $B$, where we allow $\Nonlin(z)$ to attain $\{\infty\}$ for $z\geq\sup{I}$.
This formulation of the transformed function is the one introduced in \cite{Alt-Luck1983} to define an energy functional.

Before we state the chain rule, we verify that the composition with $\psi$ behaves well with respect to the trace operator.

\begin{lemma}\label{lem:trace.interchanged.function}
	Suppose $\psi:\IR\to\IR$ is a continuous, piece-wise continuously differentiable function with bounded derivative.
	Then, $\Tr \psi(\solC)=\psi(\Tr\solC)$ in $L^2(\partial\Omega)$ for any $\solC\in H^1(\Omega)$.
\end{lemma}

\begin{proof}
	A chain rule for the composition of $\psi$ with elements in $H^1(\Omega)$ can be found in \cite{Gil-Tru98}. 
	We extend this proof to show the statement of the lemma.
	
	Since $\Omega$ is a bounded Lipschitz domain, $C^\infty(\overline{\Omega})$ is dense in $H^1(\Omega)$.
	First, suppose that $\psi$ is a $C^1$-function with bounded derivative.
	Then, following the arguments in the proof of Lemma 7.5 in \cite{Gil-Tru98}, we obtain a sequence $\{\solC_n\}_{n=1}^\infty$ in $C^\infty(\overline{\Omega})$ such that
	\[
	\solC_n\to\solC\quad \text{and} \quad \psi(\solC_n)\to \psi(\solC)\quad\text{in $H^1(\Omega)$.}
	\]
	By the continuity of the trace operator, $\solC_{n}|_{\partial\Omega}\to\Tr\solC$ and $\psi(\solC_n)|_{\partial\Omega}\to \Tr\psi(\solC)$ in $L^2(\partial\Omega)$.
	On the other hand, the Lipschitz continuity of $\psi$ implies that $\psi\left(\solC_{n}|_{\partial\Omega}\right)\to \psi(\Tr\solC)$ in $L^2(\partial\Omega)$ and hence, $\Tr \psi(\solC)=\psi(\Tr\solC)$ in $L^2(\partial\Omega)$.
	
	Next, observe that $(\solC_n)_\pm\to\solC_\pm$ and $\abs{\solC_n}\to\abs{\solC}$ in $H^1(\Omega)$ and hence, $\Tr\solC_\pm=(\Tr\solC)_\pm$ and $\Tr\abs{\solC}=\abs{\Tr\solC}$.	
	Finally, suppose that $\psi\in C(\IR)$ is only piece-wise continuously differentiable with bounded derivative.
	By an induction argument, we may assume that $\psi$ has only one corner and without loss of generality we suppose that it is at the origin. 
	We write 
	\[
	\psi=\psi_1\chi_{[0,\infty)}+\psi_2\chi_{(-\infty,0)},
	\]
	for certain $\psi_1,\psi_2\in C^1(\IR)$ with bounded derivatives.
	Then, $\psi(\solC)=\psi_1(\solC_+)+\psi_2(\solC_-)$ and the statement follows by the linearity of the trace operator.
\end{proof}

Next, we prove a chain rule that
generalizes Lemma 1.5 in \cite{Alt-Luck1983} for our setting.
In \cite{Alt-Luck1983}, the particular case $\Psi(\zeta)=\frac{1}{2}\zeta^2$ was considered which corresponds to the energy functional.
Although we do not need the chain rule in this generality for the well-posedness proof, the result might be of independent interest.
For instance, it can be applied to other Lyapunov functionals that are used in entropy methods.
Our result also includes a stronger version of Lemma 1 in \cite{Otto95} which can be shown by exploiting the structure of our equation.
The techniques we use are similar to the methods in \cite{Alt-Luck1983} and \cite{Otto95}.
In particular, property \textit{(i)} is based on Lemma 1.5 in \cite{Alt-Luck1983} and property \textit{(ii)} on Lemma 1 in \cite{Otto95}.

\begin{proposition}[Chain rule]\label{prop:chain.rule.in.time}
	Let $\solA:\Omega_T\to{I}$ be such that $\solA\in \mcW$ and $\phi(\solA)\in L^2(0,T;H^1(\Omega))$ and let $\bar{\solA}:\Omega\to{I}$ be such that $\bar{\solA}\in L^1(\Omega)$ and $\phi(\bar{\solA})\in H^1(\Omega)$.
	Let $\Psi\in C^1(\IR)$ be convex or concave and suppose that $\psi:=\Psi'$ is piece-wise continuously differentiable, $\psi'$ is bounded and define $\Psi^\star$ by \eqref{eq:transformed.function.definition}.
	Assume that
	\begin{equation}\label{eq:chain.rule.in.time.main.condition}
	\psi(\phi(\solA)-\phi(\bar{\solA}))\in \mcV.
	\end{equation}
	Finally, suppose that either
	\begin{enumerate}[(i)]
		\item $\nonlin(\bar{\solA})\in L^\infty(\Omega)\cap H^1(\Omega)$, $\psi(0)=0$ and there exists a function $\solA_0:\Omega\to{I}$ such that $\solA_0,\Psi^\star(\solA_0;\bar{\solA})\in L^1(\Omega)$ and \eqref{eq:SD-PME.solution.including.time-derivative.init.cond.id} holds, or
		\item $\psi$ is bounded. 
	\end{enumerate}
	Then, $\Psi^\star(\solA;\bar{\solA})\in L^\infty(0,T;L^1(\Omega))$ and $t\mapsto \int_{\Omega}\Psi^\star(\solA(t);\bar{\solA})$ can be represented by an absolutely continuous function $\theta$ such that
	\begin{equation}\label{eq:chain.rule.in.time}
	\theta'
	=\pinprod{\solA_t}{\psi\left(\nonlin(\solA)-\nonlin(\bar{\solA})\right)}
	\end{equation}
	a.e.\ in $(0,T)$.
	If \textit{(i)} holds, then $\theta(0)=\int_{\Omega}\Psi^\star(\solA_0;\bar{\solA})$.
\end{proposition}

\begin{proof}
	We assume without loss of generality that $\Psi$ is convex.
	If $\Psi$ is concave, then consider $-\Psi$ instead.
	To simplify notations we write 
	\[
	\solB(x,t)=\nonlin(\solA(x,t))-\nonlin(\bar{\solA}(x))\quad\text{and}\quad\Psi^\star(x,t)=\Psi^\star(\solA(x,t);\bar{\solA}(x)).
	\]
	In view of \Cref{lem:initial.cond.alternative,lem:Steklov.averaged.solution.convergence}, we set $\solA(t)=\solA_0$ for $t<0$ whenever we assume that \textit{(i)} holds.
	Moreover, we take the representative of $\solA$ in $C([0,T];(L^\infty(\Omega)\cap {V})^*)$.
	
	Fix $t\in(0,T)$.
	By \eqref{eq:transformed.function.estimates} we have 
	\begin{equation}\label{eq:transformed.function.estimates.applied.to.solution}
	[\solA(t)-\solA(t-h)]\psi(\solB(t-h))\leq \Psi^\star(t)-\Psi^\star(t-h) \leq [\solA(t)-\solA(t-h)]\psi(\solB(t))
	\end{equation}
	a.e.\ in $\Omega$.
	We write
	\[
	\lambda_\varepsilon:=\min\left\{1,\frac{1}{\varepsilon\abs{\psi(\solB)}}\right\}
	\]
	for $\varepsilon>0$, such that $\lambda_\varepsilon\abs{\psi(\solB)}=\min\{\varepsilon^{-1},\abs{\psi(\solB)}\}$.
	Note that $\lambda_\varepsilon(t)\psi(\solB(t))\in L^\infty(\Omega)\cap {V}$ and that $\lambda_\varepsilon(t)\psi(\solB(t))\to\psi(\solB(t))$ in $H^1(\Omega)$ as $\varepsilon\to 0$ by dominated convergence.
	Next, for fixed $h>0$ we have the estimates
	\begin{subequations}
		\begin{equation}\label{eq:chain.rule.first.estimate}
		\begin{aligned}
		\pinprod{\solA^h_t(t)}{\lambda_\varepsilon(t)\psi(\solB(t))}&=\int_{\Omega}\solA^h_t(t)\lambda_\varepsilon(t)\psi(\solB(t))\stackrel{\eqref{eq:Steklov.averaged.time-derivative.solution}}{=}\int_{\Omega}{\tfrac{1}{h}[\solA(t)-\solA(t-h)]}{\psi(\solB(t))\lambda_\varepsilon(t)}\\
		&\stackrel{\eqref{eq:transformed.function.estimates.applied.to.solution}}{\geq}
		\int_{\Omega}\tfrac{1}{h}\left[\Psi^\star(t)-\Psi^\star(t-h)\right]\lambda_\varepsilon(t)
		\end{aligned}
		\end{equation}
		for $t>0$ and
		\begin{equation}\label{eq:chain.rule.second.estimate}
		\begin{aligned}
		\pinprod{\solA^h_t(t)}{\lambda_\varepsilon(t-h)\psi(\solB(t-h))}&=\int_{\Omega}{\tfrac{1}{h}[\solA(t)-\solA(t-h)]}{\psi(\solB(t-h))\lambda_\varepsilon(t-h)}\\
		&\!\stackrel{\eqref{eq:transformed.function.estimates.applied.to.solution}}{\leq}
		\int_{\Omega}\tfrac{1}{h}\left[\Psi^\star(t)-\Psi^\star(t-h)\right]\lambda_\varepsilon(t-h)
		\end{aligned}
		\end{equation}
	\end{subequations}
	for $t>h$.
	
	Now, suppose \textit{(i)} holds.
	Our aim is to show that
	\begin{align}\label{eq:chain,rule.main.identity}
	\int_{\Omega}\Psi^\star(\solA(\tau),\bar{\solA})=\int_{\Omega}\Psi^\star(\solA_0;\bar{\solA})+\int_0^\tau(\solA_t(t),\psi(\solB(t)))\md t
	\end{align}
	holds for almost every $\tau\in(0,T)$ by passing to the limits $\varepsilon\to 0$ and $h\to 0$ in the estimates above.
	
	First, we show that $\Psi^\star(\solA;\bar{\solA})\in L^1(\Omega_T)$ so that we can pass to the limit $\varepsilon\to 0$ in \eqref{eq:chain.rule.first.estimate} and \eqref{eq:chain.rule.second.estimate} by dominated convergence.
	By \eqref{eq:transformed.function.positive.estimates} we have that $\Psi^\star$ is non-negative and therefore, $\int_{\Omega}\lambda_\varepsilon\Psi^\star\to \int_{\Omega}\Psi^\star$ a.e.\ in $(0,T)$ and $\int_{{\tau_1}}^{\tau_2}\int_{\Omega}\lambda_\varepsilon\Psi^\star\to\int_{{\tau_1}}^{\tau_2}\int_{\Omega}\Psi^\star$ as $\varepsilon\to 0$ for any $0\leq\tau_1\leq\tau_2\leq T$ by monotone convergence.
	Moreover, $\Psi^\star(\solA_0;\bar{\solA})\in L^1(\Omega_T)$ and hence, $\Psi^\star(\solA_0;\bar{\solA})\lambda_\varepsilon(t)\to \Psi^\star(\solA_0;\bar{\solA})$ in $L^1(\Omega)$ for almost every $t$ by dominated convergence.
	These convergences hold if we replace $\lambda_\varepsilon$ by $\lambda_\varepsilon(\emptyarg-h)$ as well.
	Rearranging the terms in \eqref{eq:chain.rule.first.estimate} and passing to the limit $\varepsilon\to 0$ we conclude that for almost every $0<t<h$ the estimate
	\[
	\frac{1}{h}\int_{\Omega}\Psi^\star(\solA(t);\bar{\solA})\leq\frac{1}{h}\int_{\Omega}\Psi^\star(\solA_0;\bar{\solA})+\pinprod{\solA^h_t(t)}{\psi(\solB(t))}
	\]
	holds.
	It follows that $\Psi^\star(\solA;\bar{\solA})\in L^1(0,h;L^1(\Omega))$.
	Repeating the argument for $h<t<2h$ we conclude that $\Psi^\star(\solA)\in L^1(0,2h;L^1(\Omega))$.
	By iteration, it follows that $\Psi^\star(\solA;\bar{\solA})\in L^1(\Omega_T)$.
	
	Now consider \eqref{eq:chain.rule.first.estimate}, let $\varepsilon\to 0$ and integrate over $t\in(0,\tau)$ for some $\tau\in [0,T]$ to obtain
	\begin{subequations}
		\begin{equation}\label{eq:chain_rule_estimate.starter.side1}
		\frac{1}{h}\int_{\tau-h}^{\tau}\int_{\Omega}\Psi^\star
		\leq \int_{\Omega}\Psi^\star(\solA_0;\bar{\solA})+\int_{0}^{\tau}\pinprod{\solA^h_t}{\psi(\solB)}.
		\end{equation}
		Also, consider \eqref{eq:chain.rule.second.estimate}, let $\varepsilon\to 0$ and integrate over $t\in(h,\tau)$ to obtain
		\begin{equation}\label{eq:chain_rule_estimate.starter.side2}
		\int_h^\tau\pinprod{\solA^h_t(t)}{\psi(\solB(t-h))}\md t\leq
		\frac{1}{h}\int_{\tau-h}^{\tau}\int_{\Omega}\Psi^\star
		-\frac{1}{h}\int_0^h\int_{\Omega}\Psi^\star(\solA;\bar{\solA}).
		\end{equation}
	\end{subequations}
	Combining \eqref{eq:chain_rule_estimate.starter.side1} and \eqref{eq:chain_rule_estimate.starter.side2} then yields
	\begin{equation}\label{eq:chain_rule_estimate}
	\begin{aligned}
	&\int_h^\tau\pinprod{\solA^h_t(t)}{\psi(\solB(t-h))}\md t+\frac{1}{h}\int_0^h\int_{\Omega}\left(\Psi^\star(\solA(t);\bar{\solA})-\Psi^\star(\solA_0;\bar{\solA})\right)\md t\\
	&\qquad\leq \frac{1}{h}\int_{\tau-h}^{\tau}\int_{\Omega}\Psi^\star
	-\int_{\Omega}\Psi^\star(\solA_0;\bar{\solA})
	\leq\int_0^{{\tau}}\pinprod{\solA_t^h(t)}{\psi(\solB(t))}\md t
	\end{aligned}
	\end{equation}
	for all $\tau\in[h,T]$.
	From \Cref{lem:Steklov.properties,lem:Steklov.averaged.solution.convergence,lem:time-shift.convergence} it follows that
	\begin{equation}\label{chain.rule.convergences.as.h->0}
	\left\{
	\begin{gathered}
	\int_{\Omega}\left[\Psi^\star\right]^h(\tau)\to\int_{\Omega}\Psi^\star(\tau)\quad\text{for almost every}\ \tau\in (0,T),\\
	\solA^h_t\to \solA_t\quad\text{in}\ L^2(0,T;{V}^*)\ \text{and}\\
	\int_{h}^T\|\psi(\solB(t-h))-\psi(\solB(t))\|_{H^1(\Omega)}^2\md t\to 0
	\end{gathered}
	\right.
	\end{equation}
	as $h\to 0$, respectively.
	Consequently, every term in \eqref{eq:chain_rule_estimate}, except for the second term, converges as $h\to 0$ for almost every $\tau\in(0,T)$, and the first and last term have the same limit.
	Moreover, by passing to the limit $h\to 0$ in \eqref{eq:chain_rule_estimate} we infer from the second estimate that $\Psi^\star(\solA;\bar{\solA})\in L^\infty(0,T;L^1(\Omega))$.
	
	The missing ingredient to conclude \eqref{eq:chain,rule.main.identity} is the nonnegativity of the second term of
	\eqref{eq:chain_rule_estimate} in the limit $h\to 0$.
	In other words, it is left to show that
	\begin{equation}\label{eq:chain_rule.notconverging.term}
	\liminf_{h\to 0}\frac{1}{h}\int_0^h\int_{\Omega}\left(\Psi^\star(\solA(t);\bar{\solA})-\Psi^\star(\solA_0;\bar{\solA})\right)\md t\geq 0.
	\end{equation}
	First, we approximate $\Psi^\star$ by
	\[
	\Psi^\star_R(z;\bar{z}):=\int_{\bar{z}}^{z}\psi_R(\nonlin(\tilde{z})-\nonlin(\bar{z}))\md \tilde{z},\quad \psi_R:=\psi\cdot\max\left\{1,\frac{R}{\abs{\psi}}\right\},\quad R>0.
	\]
	We note that $0\leq \Psi^\star_R(z;\bar{z})\leq\Psi^\star(z;\bar{z})$, since $\psi_R(0)=\psi(0)=0$ and $\abs{\psi_R}\leq\abs{\psi}$.
	It follows that $\Psi^\star_R(\solA_0;\bar{\solA})\nearrow\Psi^\star(\solA_0;\bar{\solA})$ a.e.\ in $\Omega$ as $R\to \infty$ and the convergence holds in $L^1(\Omega)$ as well by dominated convergence and since $\Psi^\star(\solA_0;\bar{\solA})\in L^1(\Omega)$.
	Next, pick $\solB_R:\Omega\to{I}$, $\solB_R\in L^\infty(\Omega)$, $\nonlin(\solB_R)\in \nonlin(\bar{\solA})+{V}$ such that $\norm{\solA_0-\solB_R}_{L^1(\Omega)}\leq R^{-2}$.
	Then, $\norm{\Psi^\star_R(\solA_0;\bar{\solA})-\Psi^\star_R(\solB_R;\bar{\solA})}_{L^1(\Omega)}\leq\norm{\psi_R}_{L^\infty(\IR)}\norm{\solA_0-\solB_R}_{L^1(\Omega)}\leq R^{-1}$ and $\psi_R(\nonlin(\solB_R)-\nonlin(\bar{\solA}))\in L^\infty(\Omega)\cap V$.
	Let $\delta>0$ be arbitrary and let $R>\frac{1}{\delta}$ be large enough such that $\norm{\Psi^\star(\solA_0;\bar{\solA})-\Psi^\star_R(\solA_0;\bar{\solA})}_{L^1(\Omega)}< \delta$. Then,
	\begin{align*}
	&\int_{\Omega}\left(\Psi^\star(\solA(t);\bar{\solA})-\Psi^\star(\solA_0;\bar{\solA})\right)\geq  \int_{\Omega}\left(\Psi^\star(\solA(t);\bar{\solA})-\Psi^\star_R(\solA_0;\bar{\solA})\right)-\delta\\
	&\geq \int_{\Omega}\left(\Psi^\star_R(\solA(t);\bar{\solA})-\Psi^\star_R(\solA_0;\bar{\solA})\right)-\delta
	\geq \int_{\Omega}\left(\Psi^\star_R(\solA(t);\bar{\solA})-\Psi^\star_R(\solB_R;\bar{\solA})\right)-2\delta\\
	&\!\!\stackrel{\eqref{eq:transformed.function.estimates}}{\geq} \int_{\Omega}\psi_R(\nonlin(\solB_R)-\nonlin(\bar{\solA}))(\solA(t)-\solB_R)-2\delta
	\geq \int_{\Omega}\psi_R(\nonlin(\solB_R)-\nonlin(\bar{\solA}))(\solA(t)-\solA_0)-3\delta
	\end{align*}
	for almost every $t$, where the choice of $R$ does not depend on $t$.
	Finally, \Cref{lem:initial.cond.alternative} and the fact that $\solA\in C([0,T];(L^\infty(\Omega)\cap{V})^*)$ implies that for fixed $R$ the term $\int_{\Omega}\psi_R(\nonlin(\solB_R)-\nonlin(\bar{\solA}))(\solA(t)-\solA_0)$ vanishes as $t\to 0$.
	Therefore, given $\delta>0$ we pick $R>\frac{1}{\delta}$ as before and set $h_0>0$ small enough such that $\abs{\int_{\Omega}\psi_R(\nonlin(\solB_R)-\nonlin(\bar{\solA}))(\solA(t)-\solA_0)}<\delta$ for all $0<t<h_0$. We conclude that
	\begin{align*}
	\frac{1}{h}\int_0^h\int_{\Omega}\left(\Psi^\star(\solA;\bar{\solA})-\Psi^\star(\solA_0;\bar{\solA})\right)&\geq -4\delta
	\end{align*}
	for all $0<h<h_0$ which implies \eqref{eq:chain_rule.notconverging.term}.
	
	We conclude that \eqref{eq:chain,rule.main.identity} holds.
	The right-hand side is continuous with respect to $\tau$ and hence, the map $\tau\mapsto\int_{\Omega}\Psi^\star(\solA(\tau),\bar{\solA})$ can be represented by an absolutely continuous function and its derivative is given by \eqref{eq:chain.rule.in.time}, as desired.
	
	Now, suppose that \textit{(ii)} holds.
	Then, we apply \eqref{eq:transformed.function.estimates} for $z_1=\solA$ and $z_2=\bar{\solA}$ and observe that $\Psi^\star(\bar{\solA};\bar{\solA})=0$ implies that $\abs{\Psi^\star(\solA;\bar{\solA})}\leq \norm{\psi}_{L^\infty(\IR)}\abs{\solA-\bar{\solA}}$. Hence, $\Psi^\star\in L^\infty(0,T;L^1(\Omega))$.
	We pass to the limit $\varepsilon\to 0$ in \eqref{eq:chain.rule.first.estimate} and \eqref{eq:chain.rule.second.estimate} for almost every $t$ using dominated convergence. Then, integrating over $t\in(\tau_1,\tau_2)$ for some $\tau_1,\tau_2\in[0,T]$, $\tau_2\geq \tau_1$ and combining the estimates yields
	\begin{equation}\label{eq:chain_rule_estimate_bounded_version}
	\int_{\tau_1}^{\tau_2}\pinprod{\solA^h_t(t)}{\psi(\solB(t-h))}\md t\leq\int_{\Omega}(\Psi^\star)^h(\tau_1)-\int_{\Omega}(\Psi^\star)^h(\tau_2)\leq\int_{\tau_1}^{\tau_2}\pinprod{\solA^h_t(t)}{\psi(\solB(t))}\md t.
	\end{equation}
	We can now pass to the limit $h\to 0$ for almost every $\tau_1,\tau_2\in(0,T)$ which is justified by the limits \eqref{chain.rule.convergences.as.h->0}.
	We conclude that
	\begin{align*}
	\int_{\Omega}\Psi^\star(\solA(\tau_2),\bar{\solA})=\int_{\Omega}\Psi^\star(\solA(\tau_1);\bar{\solA})+\int_{\tau_1}^{\tau_2}(\solA_t(t),\psi(\solB(t)))\md t
	\end{align*}
	holds for almost every $\tau_1,\tau_2\in(0,T)$ and hence, the map $\tau\mapsto\int_{\Omega}\Psi^\star(\solA(\tau),\bar{\solA})$ can be represented by an absolutely continuous function and its derivative is given by \eqref{eq:chain.rule.in.time}.
\end{proof}

We remark that the proof of \Cref{prop:chain.rule.in.time} can be simplified in the case that ${I}$ is bounded.
Indeed, then $\solA$ and $\solA_t^h$ are elements of $L^2(\Omega_T)$ and  we can 
take the $L^2$-inner product of $\solA_t^h$ with
$\psi(\nonlin(\solA)-\nonlin(\bar{\solA}))$ replacing \eqref{eq:chain.rule.first.estimate} and \eqref{eq:chain.rule.second.estimate}.
In particular, we do not need the auxiliary parameter $\varepsilon$.
Moreover, \eqref{eq:transformed.function.estimates} implies that $\Psi^\star(\solA;\bar{\solA})\in L^2(\Omega_T)$ and therefore, we can proceed with the proof as in \Cref{prop:chain.rule.in.time} \textit{(ii)}.

\begin{remark}[Linearity with respect to $\Psi$]\label{rem:chain.rule.linearity}
	The hypotheses and statements of \Cref{prop:chain.rule.in.time} are linear with respect to $\Psi$, with the exception of the assumption $\Psi^\star(\solA_0;\bar{\solA})\in L^1(\Omega)$.
	Therefore, \Cref{prop:chain.rule.in.time} holds for any linear combination of $\Psi$'s provided that $\Psi^\star(\solA_0;\bar{\solA})\in L^1(\Omega)$ holds for each $\Psi$.
	This extra assumption is trivially satisfied if $\psi$ is bounded.
	Therefore, we can apply the proposition for the sum of a convex and concave function.
\end{remark}

Observe that in the proof of \Cref{prop:chain.rule.in.time} \textit{(i)} the assumption on the initial data $\solA_0$ is essential to conclude that 
$\Psi^\star(\solA;\bar{\solA})\in L^\infty(0,T;L^1(\Omega))$.
In case \textit{(ii)} this is not necessary by the boundedness of $\psi$.
The latter case allows us to consider sub- and supersolutions.

\begin{corollary}[Chain rule for sub(super)solutions]\label{cor:chain_rule.sub(super)solution}
	Suppose the hypotheses of \Cref{prop:chain.rule.in.time} \textit{(ii)} are satisfied.
	Let $\solA_0:\Omega\to{I}$ be such that $\solA_0\in L^1(\Omega)$ and \eqref{eq:SD-PME.solution.including.time-derivative.init.cond.id} holds with $=$ replaced by $\geq$ ($\leq$) and $\testA(0)\geq 0$ instead and let $\Psi$ be convex (concave).
	Then, the conclusion of \Cref{prop:chain.rule.in.time} is valid and the absolutely continuous representative $\theta$ of $t\mapsto\int_{\Omega}\Psi^\star(\solA(t);\bar{\solA})$ satisfies
	\begin{equation}\label{eq:chain.rule.in.time.estimate.sub-supersol}
	\theta(0)\ \leq\ (\geq)\ \ \int_{\Omega}\Psi^\star(\solA_0,\bar{\solA}).
	\end{equation}
\end{corollary}

\begin{proof}
	We assume that \eqref{eq:SD-PME.solution.including.time-derivative.init.cond.id} holds with $=$ replaced by $\geq$ and $\Psi$ is convex. The other case is proven similarly.
	We adopt the notation in the proof of \Cref{prop:chain.rule.in.time}, however, 
	we do not set $\solA(t)=\solA_0$ for $t<0$.
	Since all hypotheses are satisfied, the conclusion of \Cref{prop:chain.rule.in.time} holds.
	We denote by $\int_{\Omega}\Psi^\star(\solA(0);\bar{\solA})$ the absolutely continuous representative of $t\mapsto\int_{\Omega}\Psi^\star(\solA(t);\bar{\solA})$ evaluated at $t=0$.
	
	Let $\testC\in C^\infty_c((-\infty,T))$, $\testC\geq 0$, write $\testA=\testC\psi(\solB)$ and, for simplicity, assume $\testC(0)=1$.
	Define the \textit{forward} Steklov averaged function
	\[
	\testA_h(t):=\testA^{-h}(t)=\frac{1}{h}\int_{t}^{t+h}\testA(s)\md s,
	\]
	then $\testA_h\in W^{1,1}(0,T;L^\infty(\Omega))\cap L^2(0,T;{V})$.
	We employ partial summation, that is,
	\begin{align*}
	\iint_{\Omega_T}(\solA-\solA_0)\partial_t\testA_h&=\int_h^T\int_{\Omega}\frac{\solA(t-h)-\solA_0}{h}\testA(t)\md t-\int_0^T\int_{\Omega}\frac{\solA(t)-\solA_0}{h}\testA(t)\md t\\
	&=-\int_h^T\int_{\Omega}(\partial_t\solA^h)\psi(\solB)\testC-\frac{1}{h}\int_0^h\int_{\Omega}(\solA-\solA_0)\psi(\solB)\testC\\
	&\!\stackrel{\eqref{eq:transformed.function.estimates}}{\leq}-\int_{h}^{T}\int_{\Omega}\partial_t\Psi^\star(\solA;\bar{\solA})^h\testC-\frac{1}{h}\int_0^h\int_{\Omega}\left(\Psi^\star(\solA;\bar{\solA})-\Psi^\star(\solA_0;\bar{\solA})\right)\testC.
	\end{align*}
	Since $\theta\in W^{1,2}(0,T;\mathbb{R})$, \Cref{lem:Steklov.properties} implies that $\left(\frac{\md}{\md t}\theta\right)^h=\frac{\md}{\md t}\left(\theta^h\right)$ and therefore,
	\begin{align*}
	\int_{h}^{T}\int_{\Omega}\partial_t\Psi^\star(\solA;\bar{\solA})^h\testC&=\int_{h}^{T}\int_{\Omega}\left(\Psi^\star(\solA(t);\bar{\solA})-\Psi^\star(\solA(t-h);\bar{\solA})\right)\testC(t)\md t=\int_{h}^T \frac{\md}{\md t}\left(\theta^h\right)\testC\\
	&=\int_{h}^{T}\left(\frac{\md}{\md t}\theta\right)^h\testC\stackrel{\eqref{eq:chain.rule.in.time}}{=}\int_{h}^{T}[\pinprod{\solA_t}{\psi(\solB)}]^h\testC\to \int_0^T\pinprod{\solA_t}{\psi(\solB)}\testC
	\end{align*}
	as $h\to 0$.
	By \eqref{eq:transformed.function.positive.estimates} we have that $0\leq\Psi^\star(\solA_0;\bar{\solA})\leq\norm{\psi}_{L^\infty(\Omega)}\abs{\solA_0-\bar{\solA}}$ so $\Psi^\star(\solA_0;\bar{\solA})\in L^1(\Omega)$ and thus,
	\begin{align*}
	\frac{1}{h}\int_0^h\int_{\Omega}\left(\Psi^\star(\solA;\bar{\solA})-\Psi^\star(\solA_0;\bar{\solA})\right)\testC\to \theta(0)-\int_{\Omega}\Psi^\star(\solA_0;\bar{\solA})\quad\text{as}\ h\to 0.
	\end{align*}
	Finally, from \eqref{eq:SD-PME.solution.including.time-derivative.init.cond.id} 
	we infer that
	\begin{align*}
	0&\leq\lim_{h\to 0}\int_0^T\left(\pinprod{\solA_t}{\testA_h}+\int_{\Omega}(\solA_0-\solA)\partial_t\testA_h\right)\\
	&\leq\int_0^T\pinprod{\solA_t}{\psi(\solB)}\testC-\int_0^T\pinprod{\solA_t}{\psi(\solB)}\testC+\int_{\Omega}\left(\Psi^\star(\solA_0;\bar{\solA})\right)-\theta(0)\\
	&=\int_{\Omega}\left(\Psi^\star(\solA_0;\bar{\solA})\right)-\theta(0),
	\end{align*}
	as desired.
\end{proof}

\begin{remark}[Solutions satisfy the hypotheses]\label{rem:chain.rule.checking.main.condition}
	Let $\solA:\Omega_T\to {I}$  satisfy \textit{(i)} in \Cref{def:SD-PME.solution.including.time-derivative}. 
	By \Cref{lem:trace.interchanged.function}, condition \eqref{eq:chain.rule.in.time.main.condition} is satisfied provided $\nonlin(\bar{\solA})\in \nonlin(\solA^D)+{V}$ and $\psi(0)=0$.
	Moreover, suppose $\solA$ satisfies \textit{(i)} of \Cref{def:SD-PME.sub-super.solution.including.time-derivative} in the sense of a sub(super)solution and additionally assume $\psi(\zeta)=0$ for $\zeta\leq 0\ (\geq 0)$, then \eqref{eq:chain.rule.in.time.main.condition} holds as well.
\end{remark}

\begin{remark}
	\Cref{prop:chain.rule.in.time} \textit{(ii)} is a variant of Lemma 1 in \cite{Otto95}, which we prove using the structure of our equation.
	In \cite{Otto95}, the elliptic-parabolic equation 
	\[
	\partial_t\beta(\solB)=\mathrm{div} A(\emptyarg,\solB,\nabla\solB)+f(\solB)
	\]
	is considered, where $\beta$ is only assumed to be non-decreasing.
	In particular, $\beta$ is generally not invertible. Hence, its inverse $\nonlin$ may not exists and \eqref{eq:transformed.function.definition} is not well-defined.
	Therefore, the transform
	\[
	\Psi^\star(z;\bar{\zeta})=\sup_{\zeta\in\IR}\left(\psi(\zeta-\bar{\zeta})(z-\beta(\zeta))+\int_0^\zeta\psi(\tilde{\zeta}-\bar{\zeta})\beta(\md\tilde{\zeta})\right)
	\]
	is used instead, where $\Psi^\star(\emptyarg;\bar{\zeta}):\IR\to\IR\cup\{+\infty\}$ may attain $+\infty$.
\end{remark}

Next, we discuss the absolute continuity of the relative energy functional.

\begin{corollary}[Regularity of the energy functional]\label{cor:time-regularity.energy.functional}
	Suppose $\solA^D$ and $\solA_0$ satisfy \eqref{itm:Dirichlet.boundary.condtion} and \eqref{itm:initial.condtion}, respectively, and let $\solA:\Omega_T\to{I}$ satisfy \textit{(i)} and \textit{(ii)} in \Cref{def:SD-PME.solution.including.time-derivative}.
	Then, $\Nonlin(\solA)\in L^\infty(0,T;L^1(\Omega))$ and the mapping $t\mapsto\int_{\Omega}\Nonlin(\solA(t);\solA^D)$ has an absolutely continuous representative with derivative $\pinprod{\solA_t}{\nonlin(\solA)-\nonlin(\solA^D)}$ and it attains $\int_{\Omega}\Nonlin(\solA_0;\solA^D)$ for $t=0$.
\end{corollary}

\begin{proof}
	Set $\Psi(\zeta)=\frac{1}{2}\zeta^2$.
	Then $\psi(\zeta)=\zeta$ and $\Phi(\emptyarg;\bar{z})=\Psi^\star(\emptyarg;\bar{z})$.
	Apply \Cref{prop:chain.rule.in.time} to conclude that $t\mapsto\int_{\Omega}\Nonlin(\solA(t);\solA^D)$ has an absolutely continuous representative and that $\Nonlin(\solA;\solA^D)\in L^\infty(0,T;L^1(\Omega))$.
	Now, \eqref{itm:Dirichlet.boundary.condtion} implies that $\nonlin(\solA^D)(\solA-\solA^D)\in L^\infty(0,T;L^1(\Omega))$ and therefore $\Nonlin(\solA)=\Nonlin(\solA;\solA^D)+\nonlin(\solA^D)(\solA-\solA^D)\in L^\infty(0,T;L^1(\Omega))$.
\end{proof}

The following result is based on the proof of Lemma 2 in \cite{Otto95}.
The idea is to apply \Cref{prop:chain.rule.in.time} to sequences $\{\psi_n^\pm\}_{n=1}^\infty$ that converge to $\mathrm{sign}_+$ and $\mathrm{sign}_-$, respectively, where
\[
\mathrm{sign}_+(z):=\left\{\begin{array}{cl}
1&\text{if}\ z> 0,\\
0&\text{if}\ z\leq 0,
\end{array}\right.
\quad
\mathrm{sign}_-(z):=\left\{\begin{array}{cl}
1&\text{if}\ z< 0,\\
0&\text{if}\ z\geq 0.
\end{array}\right.
\]

\begin{proposition}\label{prop:sub.super.solution.initialdata}
	Assume that \eqref{itm:nonlin.basic.assumption}, \eqref{itm:Dirichlet.boundary.condtion} and \eqref{itm:initial.condtion} hold.
	Let $\solA$ satisfy \textit{(i)} and \textit{(ii)} in \Cref{def:SD-PME.sub-super.solution.including.time-derivative} in the sense of a sub(super)solution.
	Then,
	\begin{equation}\label{eq:sub.super.solution.initialdata}
	\lim_{t\to 0}\esssup_{(0,t)}\int_{\Omega}(\solA-\solA_0)_+=0\quad\left(\lim_{t\to 0}\esssup_{(0,t)}\int_{\Omega}(\solA-\solA_0)_-=0\right).
	\end{equation}
\end{proposition}

\begin{proof}
	Let $\Psi$ be a smooth, convex function such that 
	\begin{equation}\label{eq:solution.is.continuous.mapping.convex.function.defining.properties}
	\Psi\equiv 0\quad\text{on}\quad (-\infty,0]\quad \text{and} \quad\Psi(\zeta)=\zeta-\frac{1}{2}\quad\text{for }\zeta\geq 1.
	\end{equation}
	For $\delta>0$ we set
	\[
	\Psi_\delta(\zeta)=\delta\Psi\left(\frac{\zeta}{\delta}\right).
	\]
	Further, we write $\psi:=\Psi'$ and $\psi_\delta:=\Psi_\delta'$
	Clearly, $0\leq\psi_\delta\leq \chi_{[0,\infty)}$, so from the definition of the transform in \eqref{eq:transformed.function.definition} we see that
	\begin{align*}
	0\leq (z-\bar{z})_+-\Psi^*_\delta(z;\bar{z})&=\left(\int_{\bar{z}}^z(1-\psi_\delta(\nonlin(\tilde{z})-\bar{\zeta}))\md \tilde{z}\right)_+
	=\begin{cases}
	0&\text{if}\ z\leq \bar z\ \text{or}\ \nonlin(z)\geq \bar{\zeta}+\delta,\\
	z-\bar z &\text{if}\ z\in[\bar z, \beta(\bar{\zeta}+\delta)]
	\end{cases}\\
	&\leq  \beta(\bar{\zeta}+\delta)-\bar z= \beta(\bar{\zeta}+\delta)-\beta(\bar \zeta).
	\end{align*}
	We conclude that 
	$0\leq\Psi^\star_\delta(z;\bar{z})\leq (z-\bar{z})_+$ and,
	since $\beta$ is uniformly continuous on compact intervals, it follows that
	\begin{equation}\label{eq:solution.is.continuous.mapping.properties.of.convex.function}
	\Psi^\star_\delta(z;\bar{z})\to (z-\bar{z})_+\quad\text{as $\delta\to 0$}
	\end{equation}
	for all $z,\bar{z}\in{I}$ and the convergence is uniform with respect to $z$ and $\bar{z}$ provided $\bar{z}\in{\tilde I}$ for some compact subinterval $\tilde I\subset {I}$.
	
	Now, suppose that $u$ is a function satisfying the hypotheses of the proposition as a subsolution.
	Moreover, let $\bar{\solA}:\Omega\to{I}$ be such that $\bar{\solA}\in L^1(\Omega)$ and $\nonlin(\bar{\solA})\in L^\infty(\Omega)\cap[\nonlin(\solA^D)+{V}]$.
	Then, by \Cref{rem:chain.rule.checking.main.condition} we can apply \Cref{cor:chain_rule.sub(super)solution} to obtain an absolutely continuous representative $\theta_\delta:[0,T]\to\mathbb{R}$ of $t\mapsto \int_{\Omega}\Psi^{\star}_\delta(\solA(t);\bar{\solA})$ such that and $\theta_{\delta}(0)\leq \int_{\Omega}\Psi^\star(\solA_0;\bar{\solA})$.
	The uniform convergence in \eqref{eq:solution.is.continuous.mapping.properties.of.convex.function} and using that $\nonlin(\bar{\solA})\in L^\infty(\Omega)$ imply that 
	\[
	\theta_{\delta}(t)=\int_{\Omega}\Psi^{\star}_\delta(\solA(t);\bar{\solA}) \to \int_{\Omega}(\solA(t)-\bar{\solA})_+\qquad \text{as}\ \delta\to 0,
	\]
	for almost every $t\in(0,T)$, uniformly with respect to $t$.
	Similarly, $\int_{\Omega}\Psi^\star_\delta(\solA_0;\bar{\solA})\to \int_{\Omega}(\solA_0-\bar{\solA})_+$ as $\delta\to0$.
	
	Finally, let $\{\bar u_k\}_{k=1}^\infty$ be a sequence such that 
	$\bar{\solA}_k:\Omega\to{I}$, $\bar{\solA}_k\in L^1(\Omega)$, $\nonlin(\bar{\solA}_k)\in L^\infty(\Omega)\cap[\nonlin(\solA^D)+{V}]$ and
	$\bar{\solA}_k\to \solA_0$ in $L^1(\Omega)$ as $k\to\infty$.
	Consider a sequence $\delta_k\to 0$ such that 
	\[
	\abs{\int_{\Omega}\Psi^{\star}_{\delta_k}(\solA(t);\bar{\solA}_k)- (\solA(t)-\bar{\solA}_k)_+}\leq \frac{1}{k}
	\]
	for almost every $t\in(0,T)$ and let $\theta_k$ be the absolutely continuous representative of $t\mapsto\int_{\Omega}\Psi^{\star}_{\delta_k}(\solA(t);\bar{\solA}_k)$.
	Moreover, observe that
	\begin{equation}\label{eq:solution.is.continuous.mapping.triangle_ineq-like.estimate}
	\abs{(a+b)_+-a_+}\leq \abs{b}\quad\text{for all}\ a,b\in\IR.
	\end{equation}
	Indeed, if $(a+b)_+\geq a_+$ then \eqref{eq:solution.is.continuous.mapping.triangle_ineq-like.estimate} follows from $(a+b)_+-a_+\leq b_+\leq |b|$.
	Otherwise, if $(a+b)_+< a_+$ (and therefore, $b<0<a$), then we distinguish two cases. 
	If $a\geq\abs{b}$, we have $a_+-(a+b)_+=-b=\abs{b}$ and if $\abs{b}> a$ we have $a_+-(a+b)_+=a<\abs{b}$.
	
	Therefore, we obtain
	\begin{align*}
	&\abs{\theta_{k}(t)-\int_{\Omega}(\solA(t)-\solA_0)_+}
	\leq \frac{1}{k}+\abs{\int_{\Omega}(\solA(t)-\bar{\solA}_k)_+-(\solA(t)-\solA_0)_+}\\
	&\leq \frac{1}{k}+\int_{\Omega}\abs{(\solA(t)-\solA_0+\solA_0-\bar{\solA}_k)_+-(\solA(t)-\solA_0)_+}\stackrel{\eqref{eq:solution.is.continuous.mapping.triangle_ineq-like.estimate}}{\leq} \frac{1}{k}+\int_{\Omega}\abs{\solA_0-\bar{\solA}_k}\to 0,
	\end{align*}
	as $k\to\infty$, uniformly with respect to $t$, for almost every $t\in(0,T)$.
	Moreover, recalling \eqref{eq:chain.rule.in.time.estimate.sub-supersol} we have that $\theta_k(0)\leq \int_{\Omega}\Psi^\star_{\delta_k}(\solA_0;\bar{\solA}_k)\leq\int_{\Omega}(\solA_0-\bar{\solA}_k)_+\to 0$ as $k\to \infty$.
	Let $\varepsilon>0$ be arbitrary and pick $k$ large enough such that $\abs{\theta _k(t)-\int_{\Omega}(\solA(t)-\solA_0)_+}<\varepsilon$ for almost every $t$ and $\theta_k(0) < \varepsilon$, then
	\[
	\lim_{t\to 0}\esssup_{(0,t)}\int_\Omega(\solA-\solA_0)_+< \varepsilon +\limsup_{t\to 0}\theta_k(t)=\varepsilon+\theta_k(0)\leq 2\varepsilon.
	\]
	Now \eqref{eq:sub.super.solution.initialdata} follows.
\end{proof}

\begin{corollary}\label{cor:sub.and.super.solution.initial.time.estimate}
	Assume that \eqref{itm:nonlin.basic.assumption} and \eqref{itm:Dirichlet.boundary.condtion} hold and $\solA_0$ and $\solAvar_0$ both satisfy \eqref{itm:initial.condtion}.
	Suppose that $\solA$ and $\solAvar$ satisfy \textit{(i)} and \textit{(ii)} in \Cref{def:SD-PME.sub-super.solution.including.time-derivative} in the sense of a subsolution and supersolution with respect to $\solA_0$ and $\solAvar_0$, respectively.
	Then,
	\begin{equation}
	\lim_{t\to 0}\esssup_{(0,t)}\int_{\Omega}(\solA-\solAvar)_+\leq \int_{\Omega}(\solA_0-\solAvar_0)_+.
	\end{equation}
\end{corollary}
\begin{proof}
	First, observe that the following inequality holds
	\[
	(a-b)_+\leq(a-c)_++(c-b)_+\quad\text{for all}\ a,b,c\in\IR.
	\]
	Applying this inequality twice yields
	\[
	\lim_{t\to 0}\esssup_{(0,t)}\int_{\Omega}(\solA-\solAvar)_+\leq \lim_{t\to 0}\esssup_{(0,t)}\int_{\Omega}(\solA-\solA_0)_++\int_{\Omega}(\solA_0-\solAvar_0)_++\lim_{t\to 0}\esssup_{(0,t)}\int_{\Omega}(\solA_0-\solAvar)_+.
	\]
	By \eqref{eq:sub.super.solution.initialdata}, the first and last term of the right-hand side vanish, which proves the statement.
\end{proof}

In the literature on porous medium equations, solutions $\solA\in C([0,T];L^1(\Omega))$ or $\solA\in C([0,T];L^2(\Omega))$ are often considered, e.g. see \cite{Aronson1982} and \cite{DiBenedetto2012}.
This motivates the following class of solutions. We show how it is related to the solution concept in Definition \ref{def:SD-PME.solution.including.time-derivative}.

\begin{definition}\label{def:SD-PME.solution}
	A \emph{solution} of \eqref{eq:SD-PME.Problem} is a measurable function $\solA:\Omega_T\rightarrow I$ such that
	\begin{enumerate}[label=(\roman*),font=\itshape]
		\item $\solA\in C([0,T];L^1(\Omega))$, $\nonlin(\solA)\in \nonlin(\solA^D)+ \mcV$ and
		\item the identity
		\begin{equation}\label{eq:SD-PME.solution.id}
		\displayindent0pt
		\displaywidth\textwidth
		\begin{aligned}
		\int_0^\tau\left[-\int_{\Omega}{\solA}{\testA_t}+\inprod{\nabla\nonlin(\solA)}{\nabla\testA}\right]+\int_{\Omega}{\solA(\tau)}{\eta(\tau)}=\int_{\Omega}{\solA_0}{\testA(0)}+\int_0^\tau\inprod{\source(\emptyarg,\solA)}{\eta}
		\end{aligned}
		\end{equation}
		holds for all $0\leq\tau\leq T$ and $\eta\in \mcV\cap W^{1,1}(0,T;L^\infty(\Omega))$.
	\end{enumerate}
\end{definition}

\begin{proposition}\label{prop:solution.definitions.agree}
	Suppose that \eqref{itm:nonlin.basic.assumption}, \eqref{itm:source.Lipschitz.assumption}, \eqref{itm:Dirichlet.boundary.condtion} and \eqref{itm:initial.condtion} hold.
	Further, let $\solA:\Omega_T\to{I}$ be a measurable function such that $\solA\in C([0,T];L^1(\Omega))$ and suppose that $\source(\emptyarg,\solA)\in L^2(0,T;{V}^*)$.
	Then, $\solA$ is a solution of \eqref{eq:SD-PME.Problem} in the sense of \Cref{def:SD-PME.solution.including.time-derivative} if and only if $\solA$ is a solution of \eqref{eq:SD-PME.Problem} in the sense of \Cref{def:SD-PME.solution}.
\end{proposition}

\begin{proof}
	Let $\solA:\Omega\to {I}$, $\solA\in C([0,T];L^1(\Omega))$ be a solution of \eqref{eq:SD-PME.Problem} in the sense of \Cref{def:SD-PME.solution.including.time-derivative}.
	Consider \Cref{lem:initial.cond.alternative} and $\testA_k\in C^\infty_c(\Omega)$ converging to $\mathrm{sign}(\solA(0)-\solA_0)$ in $L^1(\Omega)$ in the identity \eqref{eq:initial.cond.alternative}.
	We conclude that $\norm{\solA(0)-\solA_0}_{L^1(\Omega)}=0$, that is, $\solA(0)=\solA_0$ a.e.\ in $\Omega$.
	Next, let $\tau\in[0,T]$ and $\eta\in \mcV\cap W^{1,1}(0,T;L^\infty(\Omega))$.
	As in the proof of \Cref{lem:initial.cond.alternative} we have that
	\[
	\int_0^\tau\pinprod{\solA_t}{\eta}=\int_\Omega{\solA(\tau)}{\eta(\tau)}-\int_\Omega{\solA_0}{\eta(0)}-\int_0^\tau\int_{\Omega}{\solA}{\eta_t}.
	\]
	Substituting this into \eqref{eq:SD-PME.solution.including.time-derivative.id} shows that \eqref{eq:SD-PME.solution.id} holds.
	
	Conversely, let $\solA$ be a solution in the sense of \Cref{def:SD-PME.solution}. 
	We define the functional $\Lambda: \mcV\cap W^{1,1}(0,T;L^\infty(\Omega))\to\IR$ by
	\begin{align}\label{eq:solution.definitions.agree.RHS.Lambda}
	\Lambda(\eta)	\stackrel{\phantom{\eqref{eq:SD-PME.solution.id}}}{=}		\int_\Omega{\solA(T)}{\eta(T)}-\int_\Omega{\solA_0}{\eta(0)}-\iint_{\Omega_T} {\solA}{\eta_t}			
	\stackrel{\eqref{eq:SD-PME.solution.id}}{=}				\int_0^T\left(-\inprod{\nabla\phi(\solA)}{\nabla\eta}+\inprod{\source(\solA)}{\eta}\right).	
	\end{align}
	If $\testA$ satisfies $\norm{\eta}_{\mcV}\leq 1$, then the Cauchy-Schwarz inequality implies that 
	\[
	\abs{\Lambda(\eta)}^2\leq\norm{\nabla\nonlin(\solA)}_{L^2(\Omega_T)}^2+\norm{\source(\emptyarg,\solA)}_{L^2(0,T;{V}^*)}^2.
	\]
	This bound is independent of $\testA$, and thus $\Lambda$ is a bounded linear functional with respect to the norm of $\mcV$.
	Since the functions in $C^\infty(\overline{\Omega}\times[0,T])$ that vanish on $\Gamma$ are dense in $\mcV$, there exists a unique extension $\Lambda$ to the whole space, i.e.\ $\Lambda \in \mcV^*$.
	By \Cref{lem:duality.Bochner_spaces} this dual space is identified with $L^2(0,T;{V}^*)$ and we conclude that there exists a unique $\solB\in L^2(0,T;{V}^*)$ such that $\int_0^T\pinprod{\solB}{\eta}=\Lambda(\eta)$ for all $\eta\in \mcV$.
	In particular, it follows that 
	\begin{equation}\label{eq:solution.definitions.agree.time-derivative}
	\int_0^\tau\pinprod{\solB}{\testA}=\int_{\Omega}\solA(\tau)\testA(\tau)-\int_{\Omega}\solA_0\testA(0)-\iint_{\Omega_T}\solA\testA_t
	\end{equation}
	for all $\tau\in[0,T]$ and $\eta \in \mcV \cap W^{1,1}(0,T;L^\infty(\Omega))$. Hence, $\solB=\solA_t$ in the sense of \eqref{eq:solution.space.time.derivative} and identity \eqref{eq:SD-PME.solution.including.time-derivative.init.cond.id} is satisfied.
	Finally, \eqref{eq:SD-PME.solution.including.time-derivative.id} follows from the fact that the right-hand side of \eqref{eq:solution.definitions.agree.RHS.Lambda} is equal to $\Lambda(\testA)=\int_0^T\pinprod{\solA_t}{\testA}$, which now holds for any $\eta\in \mcV$.
\end{proof}

\section{Well-posedness of the scalar equation}\label{sect:L1-contraction}

In this section, we prove the $L^1$-contraction result and the well-posedness for problem \eqref{eq:SD-PME.Problem}, i.e. 
Theorems \ref{thm:L1-contraction} and \ref{thm:Well-posedness}. 
The proof of Theorem \ref{thm:L1-contraction} is based on the following lemma. Here, we assume that $u$ is a subsolution and $\tilde u$ a 
supersolution of \eqref{eq:SD-PME.Problem}. 
Doubling the time-variable,
\[
(x,t_1,t_2)\in \Omega\times(0,T)^2=: Q_T,
\]
we extend $\solA$ and $\solAvar$ to $Q_T$ by $\solA(x,t_1,t_2)=\solA(x,t_1)$ and $\solAvar(x,t_1,t_2)=\solAvar(x,t_2)$.
Furthermore, to shorten notations we introduce 
\[
\solC=\nonlin(\solA),\quad \solCvar=\nonlin(\solAvar),\quad \solC^D=\nonlin(\solA^D),\quad F=\source(\emptyarg,\solA),\quad \tilde{F}=\source(\emptyarg,\tilde \solA).
\]

\begin{lemma}\label{lem_L1-contraction.double.time.estimate}
	Let $u$ be a subsolution and $\tilde u$ be a supersolution of \eqref{eq:SD-PME.Problem}. Then, the inequality	
	\begin{equation}\label{eq:L1-contraction.double.time.estimate}
	\iiint_{Q_T}\left(-(\solA-\solAvar)_+(\testC_{t_1}+\testC_{t_2})-\mathrm{sign}_+(\solA-\solAvar)(F-\tilde{F})\testC\right)\leq 0
	\end{equation}
	holds for all non-negative $\testC\in C^\infty_c((0,T)^2)$.
\end{lemma}

\begin{proof}
	Let $\Psi$ be a  smooth, convex function having the properties \eqref{eq:solution.is.continuous.mapping.convex.function.defining.properties} and set
	\[
	\Psi_\delta^\pm(\zeta)=\delta\Psi\left(\pm \frac{\zeta}{\delta}\right),\quad\psi_\delta^\pm(\zeta)=\left[\Psi_\delta^\pm\right]'(\zeta)=\pm\Psi'\left(\pm\frac{\zeta}{\delta}\right).
	\]
	For fixed $t_2$, we apply \eqref{eq:SD-PME.solution.including.time-derivative.id} for the subsolution $\solA$ and the test function
	\[
	\testA(x,t_1)=\psi_\delta^+\left(\solC(x,t_1)-\tilde \solC(x,t_2)\right),
	\]
	which is justified by \Cref{rem:chain.rule.checking.main.condition}. Moreover, we can rewrite the term involving the time derivative using \eqref{eq:chain.rule.in.time} with $\bar{\solA}(x)=\solAvar(x,t_2)$,
	which implies that
	\begin{equation*}
	\frac{\md}{\md t_1}\int_{\Omega}\Psi_\delta^{+,\star}(\solA;\solAvar(t_2))+\inprod{\nabla\solC}{\nabla\psi_\delta^+\left(\solC-\solCvar(t_2)\right)}\leq\inprod{F}{\psi_\delta^+\left(\solC-\solCvar(t_2)\right)}
	\end{equation*}
	a.e.\ in $(0,T)$.
	Similarly, for fixed $t_1$ we apply \eqref{eq:SD-PME.solution.including.time-derivative.id} for the supersolution $\solAvar$ and the non-positive test function
	\[
	\testA(x,t_2)=\psi_\delta^{-}\left(\tilde\solC(x,t_2)-\solC(x,t_1)\right).
	\]
	Moreover, using \eqref{eq:chain.rule.in.time} with $\bar{\solA}(x)=u(x,t_1)$, it follows that 
	\begin{equation*}
	\frac{\md}{\md t_2}\int_{\Omega}\Psi_\delta^{-,\star}(\solAvar;\solA(t_1))+\inprod{\nabla\solCvar}{\nabla\psi_\delta^-\left(\solCvar-\solC(t_1)\right)}\leq\inprod{\tilde{F}}{\psi_\delta^-\left(\solCvar-\solC(t_1)\right)}
	\end{equation*}
	a.e.\ in $(0,T)$.
	Now, we add both inequalities, multiply the resulting inequality by a non-negative function $\testC \in C_c^\infty((0,T)^2)$ and integrate over $(t_1,t_2)\in (0,T)^2$.
	Integration by parts with respect to $t_1$ and $t_2$ then yields
	\begin{align*}
	\iiint_{Q_T}\Bigl(	&-(\Psi_\delta^{+,\star}(\solA;\solAvar)\testC_{t_1}+\Psi_\delta^{-,\star}(\solAvar;\solA)\testC_{t_2})\\
	&\quad+\left(\nabla\left(\solC-\solCvar\right)\cdot\nabla\psi_\delta^+\left(\solC-\solCvar\right)-\inprod{({F}-\tilde{F})}{\psi_\delta^+\left(\solC-\solCvar\right)}\right)\testC\Bigr)\leq 0,
	\end{align*}
	where used that $\psi_\delta^+(\zeta)=-\psi_\delta^-(-\zeta)$.
	Next, we note that
	\[
	\nabla\left(\solC-\solCvar\right)\cdot\nabla\left(\psi_\delta^+\left(\solC-\solCvar\right)\right)=\abs{\nabla\left(\solC-\solCvar\right)}^2(\psi_\delta^+)'\left(\solC-\solCvar\right)\geq 0,
	\]
	by the convexity of $\Psi^+_\delta$. Hence, we obtain
	\begin{equation}\label{eq:L1-contraction.double.time.estimate.regularized}
	\iiint_{Q_T}\left(-(\Psi_\delta^{+,\star}(\solA;\solAvar)\testC_{t_1}+\Psi_\delta^{-,\star}(\solAvar;\solA)\testC_{t_2})-\inprod{(F-\tilde{F})}{\psi_\delta^+\left(\solC-\solCvar\right)}\testC\right)\leq 0.
	\end{equation}
	The convergence in  \eqref{eq:solution.is.continuous.mapping.properties.of.convex.function}  and observing that
	\[
	\abs{\psi^+_\delta}\leq \mathrm{sign}_+,\quad \psi^+_\delta\to\mathrm{sign}_+\quad\text{point-wise as $\delta\to 0$} 
	\]
	justify that we can pass to the limit $\delta\to 0$ in \eqref{eq:L1-contraction.double.time.estimate.regularized} by dominated convergence.
	This implies \eqref{eq:L1-contraction.double.time.estimate} using that $\mathrm{sign}_+(\phi(\solA)- \phi (\tilde{\solA}))= \mathrm{sign}_+(u-\tilde u)$, which proves the lemma.
\end{proof}

We now use this lemma to prove Theorem  \ref{thm:L1-contraction} .

\begin{proof}[Proof of \Cref{thm:L1-contraction}]
	Suppose that $\solA$ is a subsolution and $\solAvar$ is a supersolution of \eqref{eq:SD-PME.Problem}. We extend $u$ and $\tilde u$ to $Q_T$ as above.
	Let $\testA\in C^\infty_c(\IR)$ denote the standard mollifier with unit mass and support in $(-1,1)$. For $\varepsilon>0$ 
	and a non-negative function $\testC\in C^\infty_c((0,T))$
	we set
	\[
	\testC_\varepsilon(t_1,t_2):=\frac{1}{\varepsilon}\testA\left(\frac{t_1-t_2}{\varepsilon}\right)\testC\left(\frac{t_1+t_2}{2}\right).
	\]
	We note that  $\testC_\varepsilon$ is an admissible test function in \eqref{eq:L1-contraction.double.time.estimate} if $\varepsilon>0$ is small enough.
	Observing that
	\[
	(\partial_{t_1}+\partial_{t_2})\testC_\varepsilon(t_1,t_2)=\frac{1}{\varepsilon}\testA\left(\frac{t_1-t_2}{\varepsilon}\right)\testC_t\left(\frac{t_1+t_2}{2}\right)
	\]
	and using the change of variables $(t,\tau)=(t_1,t_1-t_2)$, 
	implies the estimate
	\begin{equation}\label{eq:L1-contraction.estimate.regularized.+}
	\int_{\IR}\frac{1}{\varepsilon}\testA\left(\frac{\tau}{\varepsilon}\right)\iint_{\Omega_T}\left(-(\solA-\solAvar^\tau)_+\testC_t^{\tau/2}-\mathrm{sign}_+(\solA-\solAvar^\tau)(F-\tilde{F}^\tau)\testC^{\tau/2}\right)\md \tau\leq 0,
	\end{equation}
	where we write $\solA^\tau(t):=\solA(t-\tau)$.
	To pass to the limit $\varepsilon\to 0$, we study each term separately as $\tau\to 0$.
	Observe that
	\[
	\abs{(\solA-\solAvar^\tau)_+-(\solA-\solAvar)_+}=\abs{(\solA-\solAvar+\solAvar-\solAvar^\tau)_+-(\solA-\solAvar)_+}\stackrel{\eqref{eq:solution.is.continuous.mapping.triangle_ineq-like.estimate}}{\leq}\abs{\solAvar-\solAvar^\tau},
	\]
	so \Cref{lem:time-shift.convergence} and the uniform convergence of $\testC_t^{\tau/2}\to\testC_t$ show that the first term of \eqref{eq:L1-contraction.estimate.regularized.+} converges as $\tau \to 0$.
	By \eqref{itm:source.Lipschitz.assumption}, 
	it follows that
	\begin{equation}\label{eq:L1-contraction.source.term.estimates}
	\begin{aligned}
	&-\mathrm{sign}_+(\solA-\solAvar^\tau)(F-\tilde{F}^\tau)	\\
	=& -\left(\mathrm{sign}_+(\solA-\solAvar^\tau)\right)\left(\source(\emptyarg,\solA)-\sourcevar(\emptyarg,\solA)^\tau\right)
	-\left(\mathrm{sign}_+(\solA-\solAvar^\tau)\right)\left(\sourcevar(\emptyarg,\solA)^\tau-\sourcevar(\emptyarg,\solAvar)^\tau\right)\\
	\geq& -\abs{\source(\emptyarg,\solA)-\sourcevar(\emptyarg,\solA)^\tau}-\LipBoundSource(\solA^\tau-\solAvar^\tau)_+.
	\end{aligned}
	\end{equation}
	By \Cref{lem:time-shift.convergence}, we have that $\tilde{\source}(\emptyarg,\solA)^\tau\to \tilde{\source}(\emptyarg,\solA)$ in $L^2(\Omega_T)$ and $(\solA^\tau-\solAvar^\tau)_+=(\solA-\solAvar)_+^\tau\to(\solA-\solAvar)_+$ in $L^1(\Omega_T)$ as $\tau \to 0$. Therefore, the uniform convergence of $\varphi^ {\tau / 2}\to \varphi$ implies that
	\begin{equation*}
	\lim_{\tau\to 0}\iint_{\Omega_{T}}-\mathrm{sign}_+(\solA-\solAvar^\tau)(F-\tilde{F}^\tau)\testC^{\tau/2}\geq -\iint_{\Omega_{T}}\left(\LipBoundSource(\solA-\solAvar)_++\abs{\source(\emptyarg,\solA)-\sourcevar(\emptyarg,\solA)}\right)\varphi.
	\end{equation*}
	We can now pass to the limit $\varepsilon\to 0$ in \eqref{eq:L1-contraction.estimate.regularized.+} to obtain the estimate
	\begin{equation}\label{eq:L1-contraction.estimate.+}
	\iint_{\Omega_T}\left(-(\solA-\solAvar)_+\testC_t-\left(\LipBoundSource(\solA-\solAvar)_++\abs{\source(\emptyarg,\solA)-\sourcevar(\emptyarg,\solA)}\right)\varphi\right)\leq 0
	\end{equation}
	for any non-negative $\testC\in C^\infty_c((0,T))$.
	
	Fixing $s\in(0,T)$ and applying \eqref{eq:L1-contraction.estimate.+} to $\testC(t)=\testA_\varepsilon(s-t)$, where $\varepsilon>0$ is small enough, yields
	\begin{equation*}
	\int_{\Omega} \partial_t\left((\solA-\solAvar)_+\right)^\varepsilon(s)\leq\int_{\Omega} \left(\LipBoundSource(\solA-\solAvar)_++\abs{\source(\emptyarg,\solA)-\sourcevar(\emptyarg,\solA)}\right)^\varepsilon(s),
	\end{equation*}
	where $\solA^\varepsilon:=\eta_\varepsilon*\solA$ denotes the convolution of $\solA$ with $\eta_\varepsilon(t):=\frac{1}{\varepsilon}\eta(\varepsilon t)$.
	Integrating the estimate over $s\in(\varepsilon,t)$, for some $t\in(\varepsilon,T-\varepsilon)$, passing to the limit $\varepsilon\to 0$ 
	we obtain
	\begin{equation*}
	\int_{\Omega}(\solA(t)-\solAvar(t))_+\leq\int_{\Omega}(\solA_0-\solAvar_0)_++\int_0^t\int_{\Omega} \left(\LipBoundSource(\solA-\solAvar)_++\abs{\source(\emptyarg,\solA)-\sourcevar(\emptyarg,\solA)}\right)
	\end{equation*}
	for almost every $t\in(0,T)$. 
	Here, we used \Cref{cor:sub.and.super.solution.initial.time.estimate} to estimate
	\begin{align*}
	\lim_{\varepsilon\to 0}\int_{\Omega}\left((\solA-\solAvar)_+\right)^\varepsilon(\varepsilon)&\leq\lim_{\varepsilon\to 0}\sup_{\varepsilon\leq s\leq 2\varepsilon}\int_{\Omega}\left((\solA-\solAvar)_+\right)^\varepsilon(s)\\
	&=\lim_{\varepsilon\to 0}\sup_{\varepsilon\leq s\leq 2\varepsilon}\int_{-\varepsilon}^{\varepsilon}\int_{\Omega}\eta_\varepsilon(\tilde{s})[(\solA-\solAvar)_+](s-\tilde{s})\md \tilde{s}\\
	&\leq \lim_{\varepsilon\to 0}\esssup_{0<s<3\varepsilon}\int_{\Omega}(\solA-\solAvar)_+(s)\leq \int_{\Omega}(\solA_0-\solAvar_0)_+.
	\end{align*}
	
	Finally, applying Gronwall's Lemma implies \eqref{eq:L1-contraction.sub-supersolution} for $\solB=\solA$.
	
	To obtain \eqref{eq:L1-contraction.sub-supersolution} for $\solB=\solAvar$, we replace \eqref{eq:L1-contraction.source.term.estimates} by
	\begin{align*}
	&-\mathrm{sign}_+(\solA-\solAvar^\tau)(F-\tilde{F}^\tau)	\\
	=& -\left(\mathrm{sign}_+(\solA-\solAvar^\tau)\right)\left(\source(\emptyarg,\solA)-\source(\emptyarg,\solAvar)\right)
	-\left(\mathrm{sign}_+(\solA-\solAvar^\tau)\right)\left(\source(\emptyarg,\solAvar)-\sourcevar(\emptyarg,\solAvar)^\tau\right)\\
	\geq& -\LipBoundSource(\solA-\solAvar)_+ - \abs{\source(\emptyarg,\solAvar)-\sourcevar(\emptyarg,\solAvar)^\tau}.
	\end{align*}
	Again, $\sourcevar(\emptyarg,\solAvar)^\tau\to \sourcevar(\emptyarg,\solAvar)$ in $L^2(\Omega_T)$ as $\tau\to 0$.
	The other arguments are analogous as in the previous case.
	
	It remains to show \eqref{eq:L1-contraction}. 
	To this end suppose that $\solA$ and $\solAvar$ are two solutions.
	We change the proof of \Cref{lem_L1-contraction.double.time.estimate} as follows.
	Define the functions
	\[
	\Psi_\delta^{\abs{\emptyarg}}=\Psi_\delta^{+}+\Psi_\delta^{-},\quad \psi_\delta^{\abs{\emptyarg}}=\left[\Psi_\delta^{\abs{\emptyarg}}\right]'
	\]
	and observe that $|{\psi_\delta^{\abs{\emptyarg}}}|\leq 1$ 
	and that $\psi_\delta^{\abs{\emptyarg}}$ is an odd function.
	Moreover,  \eqref{eq:solution.is.continuous.mapping.properties.of.convex.function} implies that
	\[
	0\leq\Psi_\delta^{\abs{\emptyarg},\star}(z;\bar{z})\leq \abs{z-\bar{z}}\quad\text{and}\quad \Psi_\delta^{\abs{\emptyarg},\star}(z;\bar{z})\to \abs{z-\bar{z}}\quad\text{as $\delta\to 0$}
	\]
	for all $z\in{I}$.
	Furthermore, $\Psi_\delta^{\abs{\emptyarg}}$ is the sum of a convex and concave function with bounded derivatives, hence \Cref{prop:chain.rule.in.time} still holds by \Cref{rem:chain.rule.linearity}.
	
	We multiply the equations for $\solA$ and $\solAvar$ with the test functions
	\[
	\eta(x,t_1)=\psi_\delta^{\abs{\emptyarg}}(\solC(x,t_1)-\tilde\solC(x,t_2))\quad\text{and}\quad \eta(x,t_2)=\psi_\delta^{\abs{\emptyarg}}(\tilde\solC(x,t_2)-\solC(x,t_1)),
	\]
	respectively.
	Proceeding as before, we then obtain the estimate
	\begin{equation}\label{eq:L1-contraction.double.time.estimate.absolute.value}
	\iiint_{Q_T}\left(-\abs{\solA-\solAvar}(\testC_{t_1}+\testC_{t_2})-\abs{F-\tilde{F}}\testC\right)\leq 0
	\end{equation}
	for all non-negative $\varphi\in C^\infty_c(\overline{\Omega}\times(0,T))$,	which now replaces \eqref{eq:L1-contraction.double.time.estimate}.
	Inequality \eqref{eq:L1-contraction.double.time.estimate.absolute.value} 
	implies \eqref{eq:L1-contraction} using the same arguments as above in the proof of 
	\eqref{eq:L1-contraction.sub-supersolution}.
\end{proof}

To prove \Cref{thm:Well-posedness} we need the following lemma providing the well-posedness of an auxiliary elliptic problem.
\begin{lemma}\label{lem:aux.elliptic.problem}
	Let $\Omega\subset \IR^N$ be a bounded Lipschitz domain, $\Gamma\subseteq\partial\Omega$ be measurable with Hausdorff measure $H^{N-1}(\Gamma)>0$ and
	$c_1,c_2$ be positive constants. Then, there exists a unique solution 
	$\solB\in L^\infty(\Omega)\cap [c_2+{V}]$
	of the elliptic problem with mixed boundary conditions 
	\begin{equation}\label{eq:aux.elliptic.problem.for.barrier.function}
	\left\{
	\begin{aligned}
	-\Delta\solB		&=c_1	&\quad&\text{in }\Omega,\\
	\solB				&=c_2	&\quad&\text{on }\Gamma,\\
	\partial_\nu\solB	&=0		&\quad&\text{on }\partial\Omega\backslash\Gamma,
	\end{aligned}
	\right.
	\end{equation}
	and the solution satisfies $\solB\geq c_2$ a.e. in $\Omega$.
\end{lemma}

\begin{proof}
	Consider the symmetric bilinear form 
	$$
	B:V\times V\to \mathbb{R},  \qquad 
	B[\solB_1,\solB_2]=\int_\Omega\nabla\solB_1\cdot\nabla\solB_2.
	$$ 
	Then, we have  $\abs{B[\solB_1,\solB_2]}\leq \norm{\solB_1}_{H^1(\Omega)}\norm{\solB_2}_{H^1(\Omega)} \forall w_1,w_2\in V$ and, by the Poincar\'e inequality, it follows that  $\norm{\solB}^2_{H^1(\Omega)}\leq CB[\solB,\solB]\ \forall v\in V$, for some constant $C>0$. Hence, $B[\emptyarg,\emptyarg]$ defines an inner product and the induced norm is equivalent to $\|\cdot\|_{H^1(\Omega)}$.
	Therefore, Riesz's Representation Theorem implies that there exists a unique solution $\solBvar\in V$ of \eqref{eq:aux.elliptic.problem.for.barrier.function} with $c_2=0$ and $\solB:=\solBvar+c_2$ is the unique solution of \eqref{eq:aux.elliptic.problem.for.barrier.function}.
	The $L^\infty(\Omega)$-bound follows from the proof of Theorem 8.15 	in \cite{Gil-Tru98}.
	Indeed, it is based on Poincar\'e's inequality, which holds for functions in $V$.
	
	Finally, to show that $\solB\geq c_2$ we multiply \eqref{eq:aux.elliptic.problem.for.barrier.function} with the test function $\solBvar_-=(\solB-c_2)_-\in V$ which implies that
	\[
	-\norm{\nabla(\solB-c_2)_-}_{L^2(\Omega)}^2=\int_\Omega\nabla\solB\cdot\nabla(\solB-c_2)_-=\inprod{c_1}{(\solB-c_2)_-}.
	\]
	The left-hand side is non-positive and the right-hand side non-negative. Hence, $(\solB-c_2)_-=0$ in $V$ which implies that $\solB\geq c_2$ a.e.\ in $\Omega$.
\end{proof}

\begin{proof}[Proof of \Cref{thm:Well-posedness}]
	We prove each statement separately.
	\begin{enumerate}[wide, labelwidth=0pt, labelindent=0pt, topsep=0pt,itemsep=0ex,partopsep=1ex,parsep=0ex]
		\item By assumption \eqref{itm:nonlin.blow-up.assumption}, $\nonlin$ is invertible and we denote its inverse by $\beta:\IR\to {I}$.
		Note that $\beta$ is continuous and strictly increasing and hence, we can rewrite \eqref{eq:SD-PME.Problem.DiffEqn} as
		\[
		\partial_t\beta(\solC)=\Delta\solC+\source(\emptyarg,\beta(\solC)),
		\]
		where $\solC=\nonlin(\solA)$. 
		Furthermore, we introduce the notation $\solC^D=\nonlin(\solA^D)$ and $\solC_0=\nonlin(\solA_0)$.
		By Theorem 1.7 in \cite{Alt-Luck1983}, there exists a  solution $\solC:\Omega_T\to \IR$ within the class 
		\[
		\beta(\solC)\in \mcW,\quad \solC\in \solC^D+\mcV
		\]
		satisfying the identities analogous to \eqref{eq:SD-PME.solution.including.time-derivative.init.cond.id} and \eqref{eq:SD-PME.solution.including.time-derivative.id}.
		This class coincides with \Cref{def:SD-PME.solution.including.time-derivative} and hence, $\solA=\beta(\solC)$ is a solution of \eqref{eq:SD-PME.Problem}.
		The uniqueness of solutions and continuous dependence on initial data under the assumption \eqref{itm:source.growth.condition} follows from \Cref{thm:L1-contraction}.
		
		\item Let us show that \eqref{eq:energy_estimate} holds. To this end let
		$\tau\in[0,T]$ and write $\Omega_\tau=\Omega\times(0,\tau)$. We apply \eqref{eq:SD-PME.solution.including.time-derivative.id} to the test function $(\solC-\solC^D)\chi_{[0,\tau]}$ and use \Cref{prop:chain.rule.in.time} to rewrite the term involving the time derivative to obtain
		\begin{equation}\label{est_0}
		\int_{\Omega}\Nonlin(\solA(\tau);\solA^D)+\iint_{\Omega_{\tau}}\nabla\solC\cdot\nabla(\solC-\solC^D)=\int_{\Omega}{\Nonlin(\solA_0;\solA^D)}+\iint_{\Omega_{\tau}}\source(\emptyarg,\solA)(\solC-\solC^D).
		\end{equation}
		To estimate the second term from below, we use Young's inequality,
		\begin{align*}
		\iint_{\Omega_{\tau}}{\nabla\solC}\cdot{\nabla(\solC-\solC^D)}=\iint_{\Omega_{\tau}}\abs{\nabla\solC}^2-\nabla\solC\cdot\nabla\solC^D\geq\frac{1}{2}\iint_{\Omega_{\tau}}\left(\abs{\nabla\solC}^2-|{\nabla\solC^D}|^2\right).
		\end{align*}
		and the reaction term $\source(\emptyarg,\solA)(\solC-\solC^D)$ is estimated using the Cauchy-Schwarz inequality.
		We apply these estimates to \eqref{est_0} and conclude that 
		\begin{equation}\label{est_1}
		\begin{aligned}
		\norm{\Nonlin(\solA(\tau);\solA^D)}_{L^1(\Omega)}+\norm{\nabla\solC}_{L^2(\Omega_\tau)}^2
		\leq C\ \Bigl(	&\norm{\Nonlin(\solA_0;\solA^D)}_{L^1(\Omega)}+\norm{\nabla\solC^D}_{L^2(\Omega_\tau)} \\
		&+\norm{\source(\emptyarg,\solA)}_{L^2(\Omega_\tau)}\norm{\solC-\solC^D}_{L^2(\Omega_\tau)}\Bigr),
		\end{aligned}
		\end{equation}
		where the constant $C\geq 0$ is independent of the data.
		Consider \eqref{est_1} for each term on the left separately. 
		Taking the supremum over $\tau\in[0,T]$ for each estimate and adding them yields \eqref{eq:energy_estimate}.
		
		\item Let us prove that $\nonlin(\solA)=\solC$ is bounded assuming, in addition, that $\solC_0\in L^\infty(\Omega)$.
		In view of \Cref{lem:Steklov.averaged.solution.convergence} extend $\solA(t)=\solA_0$ for $t<0$.
		Suppose that $\abs{\solC_0}\leq {M}$ for some ${M}\geq 0$ large enough such that $\abs{\solC^D}\leq {M}$, which is possible by \eqref{itm:Dirichlet.boundary.condtion}.
		We define 
		\[
		c_1=\norm{\source}_{L^\infty(\Omega_T\times{I})}\qquad\text{and}\qquad c_2= {M}
		\]
		and consider the auxiliary problem \eqref{eq:aux.elliptic.problem.for.barrier.function}.
		By \Cref{lem:aux.elliptic.problem}, there exists a unique solution
		$\solB\in L^\infty(\Omega)\cap[c_2+V]$ and $\solB\geq c_2$.
		Observe that the function $(x,t)\mapsto \beta(\solB(x))$ is a supersolution of \eqref{eq:SD-PME.Problem}, because $\solB\geq c_2\geq \solC_0,\solC^D$ and
		\[
		\int_0^T\left(\pinprod{\partial_t\beta(\solB)}{\testA}+\int_\Omega\nabla\solB\cdot\nabla\testA \right)=\int_0^T\inprod{c_1}{\eta}\geq \int_0^T\inprod{\source(\emptyarg,\solB)}{\testA}
		\]
		for any non-negative test function $\eta\in\mcV$.
		Similarly, $(x,t)\mapsto \beta(-\solB(x))$ is a subsolution.
		Set $K:=\norm{\solB}_{L^\infty(\Omega)}$ and $M_0=\max\{\beta(K),-\beta(-K)\}\in{I}$, then by \Cref{thm:L1-contraction} we conclude that 
		\[
		-M_0\leq\beta(-\solB)\leq\solA\leq\beta(\solB)\leq M_0,
		\]
		that is, $\solA$ is bounded.
		Obviously, $\solC$ is bounded by $K$.
		
		\item From now on assume that \eqref{itm:nonlin.convex.differentiable.assumption} holds as well.
		Let us show that $\solC\in W^{1,2}(0,T;L^2(\Omega))$ provided that $\solC_0\in H^2(\Omega)\cap[\solC^D+{V}]$.
		Define $K':=\norm{\nonlin'(\solA)}_{L^\infty(\Omega_T)}$ and extent \eqref{eq:SD-PME.Problem.DiffEqn} for negative time by setting
		\[
		\solA(x,t)=\solA_0(x),\quad \source(x,t,z)=[\Delta\solC_0](x)\quad\text{for}\ t<0,\ x\in\Omega,\ z\in{I}.
		\]
		Indeed, \eqref{eq:SD-PME.Problem.DiffEqn} still holds, because $\solA_t$ vanishes for $t<0$ by \Cref{lem:Steklov.averaged.solution.convergence} and thus, we have $\Delta\solC_0=\source$ which, by definition, holds in the weak sense.
		
		Our aim is to multiply the equation by $\partial_t\solC$ to obtain the relevant estimates for $\solA$.
		However, $\partial_t\solC$ lacks the correct regularity to be used as a test function.
		Instead, we provide estimates for the Steklov average $\solC^h$ and pass to the limit $h\to 0$.
		First, we consider the backwards Steklov averaged version of \eqref{eq:SD-PME.Problem.DiffEqn}. 
		Fix $h>0$ and $t\in(0,T)$ and use the test function $\testA(x,s)=\frac{1}{h}\chi_{[-h,0]}(t+s)\testB(x)$ in \eqref{eq:SD-PME.solution.id} to obtain the identity
		\begin{equation}\label{eq:SD-PME.Problem.solution.id.steklov.averaged}
		\inprod{\partial_t\solA^h(t)}{\testB}+\inprod{\nabla{\solC}^h(t)}{\nabla\testB}=\inprod{\source(\emptyarg,\solA)^h(t)}{\testB}
		\end{equation}
		for any $\testB\in {V}$, for all $t\in(0,T)$.
		Note that the $({V}^*,{V})$ pairing has been replaced by the inner product, since $\solA\in L^2(\Omega_T)$ by $\solA$ being bounded and therefore, $\solA^h\in W^{1,2}(0,T;L^2(\Omega))$.
		Fix $t$ and apply \eqref{eq:SD-PME.Problem.solution.id.steklov.averaged} to the test function
		\[
		\psi=\partial_t{\solC}^h(t)=\frac{1}{h}[{\solC(t)}-{\solC(t-h)}].
		\]
		This test function is admissible, since ${\solC}(t)-\solC^D$ and ${\solC}(t-h)-\solC^D$ are in ${V}$ for almost every $t<T$.
		Indeed, we assumed that $\solC_0\in\solC^D+{V}$.
		Integrating the identity over $t\in(0,\tau)$ for some $\tau\in(0,T)$ we obtain
		\begin{equation}\label{eq:SD-PME.Problem.Regularized.steklov.averaged.ineq.uniform.bounds}
		\int_0^\tau\inprod{\partial_t\solA^h}{\partial_t{\solC}^h}+\frac{1}{2}\norm{\nabla{\nonlin(\solA(\tau))}^h}_{L^2(\Omega)}^2=\frac{1}{2}\norm{\nabla{\solC_0}}_{L^2(\Omega)}^2+\int_0^\tau \inprod{\source(\emptyarg,\solA)^h}{\partial_t{\solC}^h}.
		\end{equation}
		Young's inequality implies that  
		\begin{equation}\label{eq:SD-PME.Problem.Regularized.steklov.averaged.ineq.source.term_Young's_ineq}
		\int_0^\tau\inprod{\source(\emptyarg,\solA)^h}{\partial_t{\solC}^h}\leq C(\lambda)\LipBoundSource^2{M}_0^2\abs{\Omega_T}+\lambda\int_0^\tau \norm{\partial_t{\solC^h}}^2_{L^2(\Omega)}
		\end{equation}
		for some constant $C(\lambda)$ and arbitrary $\lambda>0$.
		By the Mean Value Theorem we have that $\abs{\partial_t\solC^h}\leq K'\abs{\partial_t\solA^h}$, for some $K'>0$, and hence,
		\begin{equation}\label{eq:SD-PME.Problem.Regularized.steklov.averaged.nonlin.MVT}
		\abs{\partial_t{\solC^h}}^2\leq K'|\partial_t\solA^h||\partial_t{\solC^h}|=K'(\partial_t\solA^h) (\partial_t{\solC}^h),
		\end{equation}
		using that $\partial_t\solA^h$ and $\partial_t\solC^h$ have the same sign.
		We estimate the first term of the left-hand side in \eqref{eq:SD-PME.Problem.Regularized.steklov.averaged.ineq.uniform.bounds} from below by \eqref{eq:SD-PME.Problem.Regularized.steklov.averaged.nonlin.MVT} and hence, setting $\lambda=\frac{1}{2K'}$ we can absorb the last term in  \eqref{eq:SD-PME.Problem.Regularized.steklov.averaged.ineq.source.term_Young's_ineq} into the left-hand side of \eqref{eq:SD-PME.Problem.Regularized.steklov.averaged.ineq.uniform.bounds}.
		We conclude that
		\begin{equation}\label{eq:well-posedness.regularized.data.final.bound}
		\norm{\partial_t{\solC^h}}^2_{L^2(\Omega_T)}+\norm{\nabla\solC^h}_{L^\infty(0,T;L^2(\Omega))}\leq C,
		\end{equation}
		for some constant $C>0$ that depends on $\norm{\nabla\solC_0}_{L^2(\Omega)}$, which is bounded since $\solC_0\in H^2(\Omega)$, but it is independent of $h$.
		Weak compactness of $L^2$-spaces implies that $\partial_t{\solC^h}$ weakly converges in $L^2(\Omega_T)$ along a subsequence, necessarily to $\partial_t\solC$, which proves that $\solC\in W^{1,2}(0,T;L^2(\Omega))$.
		Furthermore, \Cref{lem:Steklov.properties} now implies that the convergence as $h\to 0$ also holds in norm and hence, the upper bound in \eqref{eq:well-posedness.regularized.data.final.bound} holds for $\solC$ as well.
		
		Next, we prove this result for $\solC_0\in L^\infty(\Omega)\cap[\solC^D+{V}]$.
		Let $\solA_{0k}\in L^\infty(\Omega)\cap H^2(\Omega) \cap[\solC^D+{V}]$ be a sequence such that $\norm{\solA_{0k}}_{L^\infty(\Omega)}\leq \norm{\solA_{0}}_{L^\infty(\Omega)}$, $\norm{\nabla\solC_{0k}}_{L^2(\Omega)}\leq\norm{\nabla\solC_0}_{L^2(\Omega)}$ and $\solA_{0k}\to\solA_0$ in $L^1(\Omega)$.
		\Cref{thm:Well-posedness} provides the corresponding solutions $\solA$ and $\solA_k$ with $\solC_k \in W^{1,2}(0,T;L^2(\Omega))$.
		\Cref{thm:L1-contraction} shows that $\solA_k\to\solA$ in $L^\infty(0,T;L^1(\Omega))$, since the right-hand side of \eqref{eq:L1-contraction} vanishes as $k\to\infty$.
		It follows that $\solC_k\to\solC$ a.e.\ in $\Omega$ along a subsequence.
		Moreover, $\solC_k$ is bounded by ${K}$, hence $\{\solC_k\}_{k=1}^\infty$ is a bounded sequence in $L^2(\Omega_T)$ and thus, it converges weakly to $\solC$ in $L^2(\Omega)$ along a subsequence.
		Next, \eqref{eq:well-posedness.regularized.data.final.bound} provides a bound in $L^2(\Omega_T)$ for ${\partial_t\solC_k}$, uniform with respect to $k$.
		Consequently, it converges along a subsequence weakly in $L^2(\Omega_T)$, necessarily to $\partial_t\solC$.
		Therefore, $\partial_t\solC\in L^2(\Omega_T)$ as desired.	
		
		\item It remains to show that $\solA\in C([0,T];L^1(\Omega))$.
		Suppose $\nonlin'\geq\alpha>0$, then the inverse $\beta:=\nonlin^{-1}$ is a continuously differentiable function with bounded derivative and consequently, $\solA$ and $\solA^D$ are weakly differentiable with respect to space.
		Indeed, $\nabla\solA=\beta'(\solA)\nabla\nonlin(\solA)\in L^2(\Omega_T)$ and $\nabla\solA^D=\beta'(\solA^D)\nabla\nonlin(\solA^D)\in L^2(\Omega)$.
		By \Cref{lem:trace.interchanged.function}, $\nonlin\left(\Tr\solA)\right)=\Tr\nonlin(\solA)=\Tr\nonlin(\solA^D)=\nonlin\left(\Tr\solA^D\right)$ a.e.\ on $\Gamma$, and
		therefore, $\solA-\solA^D\in L^2(0,T;{V})$.
		Using that $\solA_t\in L^2(0,T;{V}^*)$, standard arguments imply that $\solA\in C([0,T];L^2(\Omega))$, see for instance Theorem 3 in Section 5 
		in \cite{evans2010}.
		Finally, by H\"older's inequality it follows that $\solA\in C([0,T];L^1(\Omega))$.
		
		Now, suppose $\nonlin'(0)=0$, $\nonlin$ is convex on ${I}\cap[0,\infty)$ and concave on ${I}\cap(-\infty,0]$, i.e.
		$\nonlin_+$ and $\nonlin_-$ are convex.
		We first prove that $\solA_{\pm}\in C([0,T];L^1(\Omega))$.
		To this end we observe that
		\begin{equation}\label{eq:superadditivity.in.reverse.triangele.ineq.}
		\nonlin_+(\abs{a-b})\leq\abs{\nonlin_+(a) - \nonlin_+(b)},\quad
		\nonlin_-(-\abs{a-b})\leq\abs{\nonlin_-(-a) - \nonlin_-(-b)}\quad\text{for all}\ a,b\geq 0,
		\end{equation}
		which follows from the convexity of $\nonlin_\pm$ and $\nonlin(0)=0$.
		Indeed, the super-additivity of $\nonlin_+$ implies that $\nonlin_+(a-b)+\nonlin_+(b)\leq\nonlin_+(a)$ if $a\geq b$, and $\nonlin_+(b-a)+\nonlin_+(a)\leq\nonlin_+(b)$ if $b\geq a$.
		A similar argument holds for $\nonlin_-$.
		From Jensen's inequality it follows that
		\begin{align*}
		&\nonlin_+\left(\fint_{\Omega}\abs{\solA_{+}(t)-\solA_{+}(s)}\right)\leq\fint_{\Omega}\nonlin_+\left(\abs{\solA_{+}(t)-\solA_{+}(s)}\right) \stackrel{\eqref{eq:superadditivity.in.reverse.triangele.ineq.}}{\leq} \fint_{\Omega}\abs{\nonlin_+(\solA_{+}(t))-\nonlin_+(\solA_{+}(s))}\\
		&=\fint_{\Omega}\abs{\solC_{+}(t)-\solC_{+}(s)} =\fint_{\Omega}\abs{(\solC(s)+\solC(t)-\solC(s))_+-\solC_{+}(s)} \stackrel{\eqref{eq:solution.is.continuous.mapping.triangle_ineq-like.estimate}}{\leq}\fint_{\Omega}(\solC(t)-\solC(s))_+\to 0
		\end{align*}
		as $s\to t$, since $\solC\in C([0,T];L^1(\Omega))$.
		Continuity of $\beta=\nonlin^{-1}$ implies that $\solA_{+}\in C([0,T];L^1(\Omega))$.
		Similarly, for $\solA_{-}$ we have
		\begin{align*}
		&\nonlin_-\left(-\fint_{\Omega}\abs{\solA_{-}(t)-\solA_{-}(s)}\right)\leq\fint_{\Omega}\nonlin_-\left(-\abs{\solA_{-}(t)-\solA_{-}(s)}\right) \stackrel{\eqref{eq:superadditivity.in.reverse.triangele.ineq.}}{\leq} \fint_{\Omega}\abs{\nonlin_-(-\solA_{-}(t))-\nonlin_-(-\solA_{-}(s))}\\
		&=\fint_{\Omega}\abs{\solC_{-}(t)-\solC_{-}(s)}=\fint_{\Omega}\abs{(\solC(s)+\solC(t)-\solC(s))_--\solC_{-}(s)} \stackrel{\eqref{eq:solution.is.continuous.mapping.triangle_ineq-like.estimate}}{\leq}\fint_{\Omega}(\solC(t)-\solC(s))_-\to 0
		\end{align*}
		as $s\to t$, and hence, $\solA_{-}\in C([0,T];L^1(\Omega))$.
		It follows that $\solA=\solA_{+}-\solA_{-}\in C([0,T];L^1(\Omega))$.
		
		Finally, we generalize this result for $\solA_0\in L^1(\Omega)$.
		Let $\solA_{0m}:\Omega\to {I}$ be a sequence such that $\solC_{0m}\in L^\infty(\Omega)\cap[\solC^D+{V}]$, $\solA_{0m}\to\solA_0$ in $L^1(\Omega)$.
		Then, \Cref{thm:Well-posedness} yields solutions $\solA_m\in C([0,T];L^1(\Omega))$ corresponding to initial data $\solA_{0m}$ and $\solA\in L^\infty(0,T;L^1(\Omega))$ corresponding to $\solA_0$.
		Moreover, \Cref{thm:L1-contraction} implies that $\{\solA_m\}_{m=1}^\infty$ converges to $\solA$ in $L^\infty(0,T;L^1(\Omega))$, since the right-hand side of \eqref{eq:L1-contraction} vanishes in the limit $m\to \infty$.
		Therefore, $\solA\in C([0,T];L^1(\Omega))$, since it is a closed subspace of $L^\infty(0,T;L^1(\Omega))$. \qedhere
	\end{enumerate}
\end{proof}

\section{Well-posedness of the coupled system}\label{sect:Well-posedness}

In this section, we prove the well-posedness for the coupled system \eqref{eq:SD-PME.coupled.Problem}.

\begin{proof}[Proof of \Cref{thm:Well-posedness.biofilm}]
	First, we extent $\nonlin$ to the open interval $(-1,1)$ anti-symmetrically, i.e.\ we set $\nonlin(-z)=-\nonlin(z)$ for $z\in (-1,0)$.
	Next, we extent ${f}$ and ${g}$ to $\Omega_T\times\IR^2$ such that \eqref{itm:source.Lipschitz.assumption.coupled} holds (e.g., set ${f}(\emptyarg,u,\emptyarg)={f}(\emptyarg,0,\emptyarg)$ for $u<1$, ${f}(\emptyarg,u,\emptyarg)=\lim_{z\nearrow 1}{f}(\emptyarg,z,\emptyarg)$ for $u\geq 1$, etc.).
	Note that ${f}$ and ${g}$ are bounded by $2\LipBoundSource$.
	We define the spaces
	\[
	\begin{aligned}
	\mcX&=\{\solA:\Omega_T\to(-1,1) :\ \solA\in \mcW_1,\ \nonlin(\solA)\in \nonlin(\solA^D)+\mcV_1,\ \text{\eqref{eq:SD-PME.solution.including.time-derivative.init.cond.id} holds with respect to}\ \solA_0\},\\
	\mcY&=\left\{\solB\in \mcW_2\cap [\solB^D+\mcV_2] :\ \solB(0)=\solB_0\right\},
	\end{aligned}
	\]
	where we observe that $\mcW_2\cap[\solB^D+\mcV_2]\subset C([0,T];L^2(\Omega))$. Therefore, the condition $\solB(0)=\solB_0$ is well-defined and equivalent with  \eqref{eq:SD-PME.solution.including.time-derivative.init.cond.id} satisfied by $\solB_0$.
	Indeed, given $\solB\in \mcW_2\cap[\solB^D+\mcV_2]$ we have that $\solBvar:=\solB-\solB^D\in\mcV$ and $\solBvar_t\in L^2(0,T;{V}_2^*)$.
	Standard arguments then imply that $\solBvar\in C([0,T];L^2(\Omega))$, see for example Theorem 3 in Section 5.9 
	of \cite{evans2010}.
	
	Let $\mcT_1:\mcX\to\mcY$ denote the mapping that assigns to $\solA\in\mcX$ the solution $\solB\in \mcY$ of the second equation in \eqref{eq:coupled.SD-PME.problem.diff.eq}. It is well-defined since a unique solution $v$ exists by \Cref{thm:Well-posedness}, substituting $I=\IR$, $\nonlin(z)=z$ and $\source(\emptyarg,\solB)={g}(\emptyarg,\solA,\solB)$.
	Similarly, let $\mcT_2:\mcY\to\mathcal{{X}}$ denote the mapping that assigns to $\solB\in\mcY$ the solution $\solA\in \mcX$ of the first equation in \eqref{eq:coupled.SD-PME.problem.diff.eq}. 
	A unique solution exists by \Cref{thm:Well-posedness}, taking $I=(-1,1)$ and $\source(\emptyarg,\solA)={f}(\emptyarg,\solA,\solB)$.
	Observe that by \Cref{thm:L1-contraction}, $\mcT_1$ and $\mcT_2$ are continuous with respect to the $L^\infty([0,T];L^1(\Omega))$-norm.
	Indeed, by \eqref{eq:L1-contraction} we have the estimates
	\begin{align*}
	\norm{\mcT_1(\solA)-\mcT_1(\solAvar)}_{L^\infty([0,T];L^1(\Omega))}
	&\leq e^{\LipBoundSource T}\norm{{g}(\emptyarg,\solA,\mcT_1(\solA))-{g}(\emptyarg,\solAvar,\mcT_1(\solA))}_{L^1(\Omega_T)}\\
	&\leq \LipBoundSource T e^{\LipBoundSource T}\norm{\solA-\solAvar}_{L^\infty([0,T];L^1(\Omega))}
	\end{align*}
	and
	\begin{align*}
	\norm{\mcT_2(\solB)-\mcT_2(\solBvar)}_{L^\infty([0,T];L^1(\Omega))}\leq \LipBoundSource T e^{\LipBoundSource T}\norm{\solB-\solBvar}_{L^\infty([0,T];L^1(\Omega))}.
	\end{align*}
	In particular, for $T>0$ small enough such that $LTe^{\LipBoundSource T}<1$, the mappings are contractions and therefore, 
	\[
	\mcA:=\mcT_2\circ\mcT_1:\mcX\to\mcX
	\]
	is a contraction with respect to the $L^\infty([0,T];L^1(\Omega))$-norm as well.
	Let $\overline{\mcX}$ denote the completion of $\mcX$ with respect to the $L^\infty([0,T];L^1(\Omega))$-topology and extend $\mcA$ to $\overline{\mcX}$ continuously.
	By Banach's Fixed Point Theorem, $\mcA$ has a unique fixed point $\bar{\solA}\in\overline{\mcX}$ and it is the $L^\infty([0,T];L^1(\Omega))$-limit of a sequence $\{\solA_k\}_{k=1}^\infty$ in $\mcX$ given by 
	\[
	\solA_{k}=\mcA^k(\solA)
	\]
	for some $\solA\in\mcX$.
	
	We write $\solB_k=\mcT_1(\solA_k)$ and $\bar{\solB}=\mcT_1(\bar{\solA})$.
	By \eqref{itm:source.positivity.coupled}, zero is a subsolution for both equations in \eqref{eq:SD-PME.coupled.Problem} and hence, \Cref{thm:L1-contraction} implies that $\solA_k\geq 0$ and $\solB_k\geq 0$ a.e.\ in $\Omega_T$.
	Moreover, \eqref{itm:source.positivity.coupled} implies that the constant $1$ is a supersolution for the second equation and we conclude that $\solB_k\leq 1$ a.e.\ in $\Omega_T$.
	Since $\solA_k\to \bar{\solA}$ and $\solB_k\to\bar{\solB}$ a.e.\ in $\Omega_T$ along a subsequence, we conclude that $\bar{\solA}\geq 0$ and $0\leq\bar{\solB}\leq 1$ a.e.\ in $\Omega_T$.
	Furthermore, using these bounds we infer from H\"older's inequality and the boundedness of $\Omega$ that 
	\begin{equation}\label{eq:coupled_equation_strong_L2-convergence}
	\text{$\solA_k\to\bar{\solA}$ and $\solB_k\to\bar{\solB}$ in 
		$L^\infty(0,T;L^2(\Omega))$.}
	\end{equation}
	
	To show that $(\bar{\solA},\bar{\solB})\in\mcX\times\mcY$, we use energy estimates and derive bounds for $\solA_k$ that are independent of $k$.
	For each $\solA_k$ estimate \eqref{eq:energy_estimate} holds, so by \Cref{rem:energy_uniform_bound} we see that
	\begin{equation}\label{eq:energy_estimate_coupled_uniform_bound.1}
	\begin{aligned}
	&\norm{\Nonlin(\solA_k)}_{L^\infty(0,T;L^1(\Omega))}+\norm{\nabla\nonlin(\solA_k)}_{L^2(\Omega_T)}^2\leq {C}_1,
	\end{aligned}
	\end{equation}
	for some constant ${C}_1>0$.
	By exactly the same arguments applied to the second equation in \eqref{eq:coupled.SD-PME.problem.diff.eq} with $\nonlin(z)=z$ and $\Nonlin(z)=\frac{1}{2}z^2$, we conclude that
	\begin{equation}\label{eq:energy_estimate_coupled_uniform_bound.2}
	\norm{\solB_k}_{L^\infty([0,T];L^2(\Omega))}+\norm{\nabla\solB_k}^2_{L^2(\Omega_T)}\leq {C}_2,
	\end{equation}
	for some constant ${C}_2>0$.
	Using that $L^2$-spaces are compact in the weak topology and that limits are unique we infer from \eqref{eq:energy_estimate_coupled_uniform_bound.1} and \eqref{eq:energy_estimate_coupled_uniform_bound.2} that
	\begin{equation}\label{eq:coupled_equation_weak_convergence_along_subseq}
	\text{$\nonlin(\solA_k)\rightharpoonup\nonlin(\bar{\solA})$ and $\solB_k\rightharpoonup\bar{\solB}$ weakly in $L^2(0,T;H^1(\Omega))$ along a subsequence.}
	\end{equation}
	
	We further observe that \eqref{eq:coupled_equation_strong_L2-convergence} and \eqref{itm:source.Lipschitz.assumption.coupled} imply 
	that
	\begin{equation*}
	{f}(\emptyarg,\solA_k,\solB_k)\to {f}(\emptyarg,\bar{\solA},\bar{\solB})\quad \text{in $L^\infty(0,T;L^2(\Omega))$},
	\end{equation*}
	and hence, $\norm{{f}(\emptyarg,\solA_k,\solB_k)}_{L^\infty(0,T;L^2(\Omega))}$  is uniformly bounded with respect to $k$.
	Using the solution identity \eqref{eq:SD-PME.solution.id} for both equations and the boundedness of weakly convergent sequences, we conclude that $\{\partial_t\solA_k\}_{k=1}^\infty$ and $\{\partial_t\solB_k\}_{k=1}^\infty$ are bounded sequences in $L^2(0,T;{V}_1^*)$ and $L^2(0,T;{V}_2^*)$, respectively.
	Indeed, for any $\testA\in\mcV_1$ with $\norm{\testA}_{L^2(0,T;H^1(\Omega))}\leq 1$ we have that 
	\[
	\int_0^T\pinprod{\partial_t\solA_k}{\testA}=\int_0^T-\inprod{\nabla\nonlin(\solA_k)}{\nabla\testA}+\inprod{\source(\emptyarg,\solA_k,\solB_k)}{\testA}\leq C,
	\]
	for some constant $C,$ by the Cauchy-Schwarz inequality.
	The same argument holds for $\solB_k$.
	Since reflexive spaces are compact in the weak topology and limits are unique, we obtain 
	\begin{equation}\label{eq:coupled_equation_weak_convergence_time_derivative}
	\int_0^T\pinprod{\partial_t\solA_k}{\testA}\to\int_0^T\pinprod{\solA_t}{\testA}\ \text{and}\ \int_0^T\pinprod{\partial_t\solB_k}{\testB}\to\int_0^T\pinprod{\solB_t}{\testB}
	\end{equation}
	for all $\testA\in\mcV_1$ and $\testB\in \mcV_2$ along a subsequence.
	
	The convergences in \eqref{eq:coupled_equation_weak_convergence_along_subseq} and \eqref{eq:coupled_equation_weak_convergence_time_derivative} are sufficient to conclude that $(\bar{\solA},\bar{\solB})\in {\mcX}\times{\mcY}$.
	Indeed, by continuity of the trace operator the sets $\nonlin(\solA^D)+\mcV_1$ and $\solB^D+\mcV_2$ are closed in $H^1(\Omega)$.
	The sets are convex as well, so Mazur's Lemma implies that they are weakly closed.
	Therefore, by \eqref{eq:coupled_equation_weak_convergence_along_subseq} we conclude that $\bar{\solA}\in\nonlin(\solA^D)+\mcV_1 $ and $\bar{\solB}\in \solB^D+\mcV_2$.
	Further, \eqref{eq:coupled_equation_weak_convergence_time_derivative} shows weak convergence in $\mcW_1$ and $\mcW_2$ and hence, $\solA_t\in \mcW_1$ and $\solB_t\in \mcW_2$.
	
	Now, $\bar{\solA}$ is the fixed point of $\mcA$, hence $(\bar{\solA},\bar{\solB})$ is a solution of system \eqref{eq:SD-PME.coupled.Problem}. 	
	Uniqueness and the continuous dependence on initial data in the $L^1$-topology follows from \Cref{thm:L1-contraction} applied to both equations in \eqref{eq:coupled.SD-PME.problem.diff.eq} separately and adding the resulting inequalities.
	Similarly, the energy estimate follows from \eqref{eq:energy_estimate} applied to both equations separately and adding the two estimates.
	
	From now on, let $(\solA,\solB)$ denote the solution of \eqref{eq:SD-PME.coupled.Problem}.	
	To conclude that $\mcA$ is a contraction we assumed that $T$ is small enough.
	The restriction on $T$ only depends on $\LipBoundSource$ and therefore, we can remove this requirement by pasting together solutions defined on sufficiently small time intervals of length $\delta T$.
	
	Finally, suppose $\nonlin$ satisfies \eqref{itm:nonlin.convex.differentiable.assumption}, then \Cref{thm:Well-posedness} provides solutions $\solA_k\in C([0,T];L^1(\Omega))$ and hence, $\solA$ also lies in this space since it is a closed subspace of $L^\infty(0,T;L^1(\Omega))$.	
	The final statement of \Cref{thm:Well-posedness.biofilm} is a direct consequence of the final statement of \Cref{thm:Well-posedness} which completes the proof.
\end{proof}

\appendix

\section{Appendix}\label{sect:Appendix}

We prove several properties of Bochner spaces and Steklov averages. 

\begin{proposition}[Fundamental Theorem of Calculus in Bochner spaces]\label{prop:Fund.Thm.Cal.Bochner_spaces}
	Let $X$ be a Banach space, $\solA\in L^1_\loc(\IR,{X})$ and consider the following statements:
	\begin{enumerate}[(i), font=\itshape]
		\item there exists a function $\solB\in L^1_\loc(\IR;{X})$ such that
		\begin{equation}\label{eq:absolutely.continuous.Fundamental.Integral}
		\solA(t)=\solA(0)+\int_{0}^{t}\solB(s)\md s\quad\text{for all $t\in\IR$;}
		\end{equation}
		\item $\solA\in W^{1,1}_\loc(\IR;{X})$;
		\item $\solA$ is almost everywhere differentiable with $\solA'\in L^1_\loc(\IR;{X})$.
		\setcounter{listCounter}{\value{enumi}}
	\end{enumerate}
	Then, the function $\solA$ satisfies \textit{(i)} if and only if $\solA$ satisfies \textit{(ii)}.
	In this case, $\solA$ is locally absolutely continuous, \textit{(iii)} holds and the almost everywhere derivative and the weak derivative coincide with $\solB$ a.e.\ in $\IR$.
	
	If $\solA$ is locally absolutely continuous, then \textit{(i)-(iii)} are equivalent. 
	If $X$ is reflexive, then $\solA$ is locally absolutely continuous if and only if \textit{(i)-(iii)} hold.
\end{proposition}

\begin{proof}
	Let $\solA\in L^1_\loc(\IR;{X})$, then Lemma 2.5.8 and Proposition 2.5.9 and their proofs 
	in \cite{Hyt-Neerv-Ver-Weis2016} show that \textit{(i)} holds if and only if \textit{(ii)} holds and that both imply \textit{(iii)}.
	Proposition 2.5.9 also shows that if $\solA$ is locally absolutely continuous, then \textit{(i)-(iii)} are equivalent.
	
	Left to show is that \textit{(i)-(iii)} hold if $\solA$ is locally absolutely continuous and $X$ is reflexive.	
	Now, Corollary 2.5.15 
	in \cite{Hyt-Neerv-Ver-Weis2016} shows that, in this case, \textit{(ii)} holds.
	From the first part of \Cref{prop:Fund.Thm.Cal.Bochner_spaces} it follows that \textit{(i)} and \textit{(iii)} hold as well.
\end{proof}

\begin{lemma}[Duality of Bochner spaces]\label{lem:duality.Bochner_spaces}
	Let $X$ be a reflexive Banach space, then
	\[
	L^p(0,T;X^*)\cong L^q(0,T;X)^*,\quad1<p,q<\infty,\quad\frac{1}{p}+\frac{1}{q}=1,
	\]
	where the isometric isomorphism is given by 
	\[
	\solA\mapsto\left(\solB\mapsto\int_0^T\pinprod{\solA(t)}{\solB(t)}_{X^\ast,X}\md t\right).
	\]
\end{lemma}

\begin{proof}
	See Corollary 1.3.22 
	in \cite{Hyt-Neerv-Ver-Weis2016}.
\end{proof}

\begin{lemma}[Properties of Steklov averaging]\label{lem:Steklov.properties}
	Let $X$ be a Banach space and $\solA\in L^1_\loc(\IR;X)$.
	Then, $\solA^h\in W^{1,1}_\loc(\IR;X)$ with $\partial_t\solA^h(t)=\frac{1}{h}[\solA(t)-\solA(t-h)]$ for almost every $t\in\IR$.
	\begin{enumerate}[(i), font=\itshape]
		\item
		If $\solA\in C(\IR;{X})$, then $\solA^h\to\solA$ in $C([a,b];X)$ for every compact interval 
		$[a,b]\subset\IR$.
		\item
		If $\solA\in L^p_\loc(\IR;{X})$, $p\in[1,\infty)$, then $\solA^h\to\solA$ in $L^p_\loc(\IR;{X})$ and $\solA^h(t)\to\solA(t)$ in $X$ for almost every $t\in\IR$.
		\item
		If $\solA\in W^{1,1}_\loc(\IR;X)$, then $\partial_t\solA^h=(\solA_t)^h$ and $\solA^h\to\solA$ in $W^{1,1}_\loc(\IR;X)$.
		\item
		If $\solA\in L^2(0,T;H^1(\Omega))$, then $(\nabla\solA)^h=\nabla\solA^h$.
		\setcounter{listCounter}{\value{enumi}}
	\end{enumerate}
\end{lemma}

\begin{proof}
	To show the first statement we compute
	\begin{align*}
	\frac{1}{h}\int_0^t\left(\solA(s)-\solA(s-h)\right)\md s&=\frac{1}{h}\int_0^t\solA(s)\md s-\frac{1}{h}\int_{-h}^{t-h}\solA(s)\md s\\
	&=\frac{1}{h}\int_{t-h}^t\solA(s)\md s-\frac{1}{h}\int_{-h}^0\solA(s)\md s=\solA^h(t)-\solA^h(0),
	\end{align*}
	which proves \eqref{eq:absolutely.continuous.Fundamental.Integral}. The claim now follows using \Cref{prop:Fund.Thm.Cal.Bochner_spaces}.	
	To show statement \textit{(i)}  we assume $\solA$ to be continuous and consider any compact subinterval $[a,b]\subset\IR$.
	Let $\varepsilon>0$ and pick $\delta>0$ such that $\norm{\solA(t)-\solA(s)}_X<\varepsilon$ for all $t,s\in[a,b]$, $\abs{t-s}<\delta$ and let $0<h<\delta$. Then, $\norm{\solA^h(t)-\solA(t)}_X\leq\frac{1}{h}\int_{t-h}^t\norm{\solA(s)-\solA(t)}_X\md s<\varepsilon$ for all $t\in[a,b]$, so $\solA^h\to \solA$ uniformly on $[a,b]$ as $h\to 0$.
	
	To prove \textit{(ii)} for a general $\solA\in L^p_\loc(\IR)$ we use the denseness of $C_c(\IR;X)$ in $L^p(\IR;X)$, see Lemma 1.2.31 
	in \cite{Hyt-Neerv-Ver-Weis2016}, in the following $\frac{\varepsilon}{3}$-argument: fix $t_1\leq t_2$, $0<\varepsilon<1$ and pick $\solB\in C_c(\IR;X)$ such that $\int_{t_1-1}^{t_2+1}\norm{\solA-\solB}_X^p\leq\frac{\varepsilon}{3}$ and $0<\delta<1$ such that $\int_{t_1-1}^{t_2+1}\norm{\solB^h-\solB}_X^p\leq\frac{\varepsilon}{3}$ for all $0<h<\delta$. Then,
	\[
	\int_{t_1}^{t_2}\norm{\solA^h-\solA}_X^p\leq\int_{t_1}^{t_2}\left(\norm{\solA^h-\solB^h}_X^p+\norm{\solB^h-\solB}_X^p+\norm{\solB-\solA}_X^p\right)\leq \frac{\varepsilon}{3}+\frac{\varepsilon}{3}+\frac{\varepsilon}{3},
	\]
	where we used that $\solA^h-\solB^h=(\solA-\solB)^h$ and that for any $\solC\in L^1_\loc(\IR;X)$ we have that
	\begin{align*}
	\int_{t_1}^{t_2}\norm{\solC^h}_X^p
	&\leq\frac{1}{h}\int_{t_1}^{t_2}\int_{t-h}^{t}\norm{\solC(s)}_X^p\md s\md t=\frac{1}{h}\int_{t_1}^{t_2}\left(\int_{t_1}^{t}\norm{\solC(s)}_X^p\md s-\int_{t_1}^{t-h}\norm{\solC(s)}_X^p\md s\right)\md t\\
	&=\frac{1}{h}\left(\int_{t_1}^{t_2}\int_{t_1}^t\norm{\solC(s)}_X^p\md s\md t-\int_{t_1-h}^{t_2-h}\int_{t_1}^{t}\norm{\solC(s)}_X^p\md s\md t\right)\\
	&=\frac{1}{h}\left(\int_{t_2-h}^{t_2}\int_{t_1}^t\norm{\solC(s)}_X^p\md s\md t-\int_{t_1-h}^{t_1}\int_{t_1}^{t}\norm{\solC(s)}_X^p\md s\md t\right)\leq \int_{t_1-1}^{t_2+1}\norm{\solC(t)}^p\md t.
	\end{align*}
	Therefore, $\int_{t_1}^{t_2}\norm{\solA^h-\solB^h}_X^p\leq \frac{\varepsilon}{3}$.
	Consequently, $\solA^h\to\solA$ in $L^p_\loc(\IR;X)$.
	The point-wise convergence is given by Theorem 2.3.4 and Corollary 2.3.5 
	in \cite{Hyt-Neerv-Ver-Weis2016}.
	
	Statement \textit{(iii)}  is easily shown using the fact that \eqref{eq:absolutely.continuous.Fundamental.Integral} is satisfied for $\solB=\solA_t$, and hence,
	\[
	\partial_t\solA^h(t)=\frac{1}{h}(\solA(t)-\solA(t-h))=\frac{1}{h}\int_{t-h}^t \solA_t(s)\md s =(\solA_t)^h(t)	\quad \text{for almost every}\ t\in(0,T).
	\]
	It follows that $\partial_t\solA^h=\solA_t^h\to \solA_t$ in $L^1_\loc(\IR;X)$.
	Finally, \textit{(iv)}  is verified using \Cref{rem:Steklov.point-wise.value} and Fubini's Theorem twice: let $\testA\in C^\infty_c(\Omega)$, then
	\begin{align*}
	&\int_{\Omega}(\nabla\solA)^h(t)\testA
	=\frac{1}{h}\int_{\Omega}\int_{t-h}^t\nabla\solA(x,s)\eta(x)\md s\md x=\frac{1}{h}\int_{t-h}^t\int_{\Omega}\nabla\solA(x,s)\eta(x)\md x\md s\\
	&\quad=-\frac{1}{h}\int_{t-h}^{t}\int_{\Omega}\solA(x,s)\nabla\eta(x)\md x\md s=-\frac{1}{h}\int_{\Omega}\int_{t-h}^{t}\solA(x,s)\nabla\eta(x)\md s\md x=-\int_{\Omega}\solA^h(t)\nabla\testA,
	\end{align*}
	as desired.
\end{proof}

\begin{lemma}[Time-shifted functions]\label{lem:time-shift.convergence}
	Let $X$ be a Banach space and $\solA\in L^p_\loc(\IR;X)$, $p\in[1,\infty)$, then for any $t_2\geq t_1$ we have
	\[
	\int_{t_1}^{t_2}\norm{\solA(t)-\solA(t-h)}_X^p\md t\to 0
	\]
	as $h\to 0$.
\end{lemma}

\begin{proof}
	The statement follows from the fact that $C_c(\IR;X)$ is dense in $L^p(\IR;X)$, by Lemma 1.2.31 
	in \cite{Hyt-Neerv-Ver-Weis2016}, because it allows to use an analogous argument as provided in Lemma 4.3 
	in \cite{Brezis2010}.
\end{proof}
	
\noindent \textbf{Acknowledgement.} We thank K.\ Mitra for bringing the articles \cite{Alt-Luck1983} and \cite{Otto95} to our attention.
\\[2ex]
Funding: The second author is supported by the NWO grant OCENW.KLEIN.358.
	
\bibliographystyle{abbrv}
\bibliography{bibliography}
\end{document}